\documentclass{article}

\usepackage[utf8]{inputenc}
\usepackage[T1]{fontenc}
\usepackage{amssymb}
\usepackage[english]{babel}
\usepackage{amsmath}
\usepackage{amsthm}
\usepackage{epsfig}
\usepackage{mathrsfs}
\usepackage{xcolor}
\usepackage{graphicx}
\usepackage{enumerate}
\usepackage{cases}
\usepackage[toc,page]{appendix}
\usepackage{hyperref}
\usepackage{xcolor}
\usepackage{verbatim}

\usepackage{todonotes}\reversemarginpar

\newtheorem{theorem}{Theorem}[section]
\newtheorem{lemma}[theorem]{Lemma}

\newtheorem{proposition}[theorem]{Proposition}
\newtheorem{corollary}[theorem]{Corollary}

\newtheorem{assumption}{Assumption}

\newcommand{\regdelta}{R_{\delta}}
\newcommand{\regdoubledelta}{R_{\delta}}
\newcommand{\Rd}{R_{\delta}}

\newcommand{\xp}{x_{\parallel}}
\newcommand{\xo}{x_\perp}
\newcommand{\pl}{\parallel}
\newcommand{\pe}{\perp}

\def\eps{\varepsilon}
\def\oeps{\eps}
\def\opsi{K}

\def\pa{\partial}
\def\na{\nabla}
\def\RR{\mathbb R}
\def\R{\mathbb R}

\def\vphi{\varphi}

\title{Local regularity result for an 
optimal transportation problem with rough measures in the plane}
\author{P.-E. Jabin\thanks{pejabin@psu.edu, Pennsylvania State University, Department of Mathematics and Huck Institutes, State College, PA 16802. Partially supported by NSF DMS Grant 161453, 1908739, 2049020 and NSF Grant RNMS (Ki-Net) 1107444.}, A. Mellet\thanks{mellet@umd.edu, University of Maryland, Department of Mathematics and CSCAMM, College Park, MD 20742  USA. Partially supported by NSF Grant NSF Grant DMS-1501067 and DMS-2009236.}, M. Molina-Fructuoso \thanks{mmolina2@ncsu.edu, North Carolina State University, Department of Mathematics, Raleigh, NC 27695 }}
\date{}

\begin{document}
\maketitle
\begin{abstract}
We investigate the properties of convex functions in $\RR^2$ that satisfy a local inequality which generalizes the notion of sub-solution of Monge-Amp\`ere equation for a Monge-Kantorovich problem with quadratic cost between {\em non-absolutely} continuous measures.  
	For each measure, we introduce a discrete scale so that the measure behaves as an absolutely continuous measure up to that scale. Our main theorem then proves that such convex functions cannot exhibit any flat part at a scale larger than the corresponding discrete scales on the measures. This, in turn, implies a $C^1$ regularity result up to the discrete scale for the Legendre transform. Our result applies in particular to any Kantorovich potential associated to an optimal transportation problem between two measures that are (possibly only locally) sums of uniformly distributed Dirac masses.
	The proof relies on novel explicit estimates directly based on the optimal transportation problem, instead of the Monge-Amp\`ere equation. 
\end{abstract}

\section{Introduction}
\subsection{Generalized Monge-Amp\`ere equation for rough measures}
%
%
%

In this paper, we investigate the properties of convex functions in $\RR^2$ that can be seen as {\em local} one-sided Kantorovich potentials. More specifically, we consider a (continuous) convex function $\psi:\RR^2\to \RR$  that satisfies 
\begin{equation} \label{eq:ineq1}
\mu(A) \leq \nu (\pa\psi(A))\quad \mbox{ for all Borel sets $A\subset \Omega$}
\end{equation}
where  $\mu$ and $\nu$ are two non-negative Radon measures on $\RR^2$ and $\Omega$ is an open bounded subset of $\RR^2$. 

We recall that the subdifferential $\pa \psi $ of the function $\psi$ is given by
$$
\pa\psi(x) = \left\{ z\in \RR^2\, ;\, \psi(y)\geq \psi(x)+z\cdot(y-x) \mbox{ for all } y\in\RR^2\right\}
$$
and that  $\pa\psi(A) = \cup_{x\in A} \pa\psi(x)$.

We note that  \eqref{eq:ineq1} is a {\em local condition} as it is only posed in  a subset $\Omega\subset \RR^2$. 
As we will explain shortly, 
Inequality \eqref{eq:ineq1} appears naturally when $\psi$ is a Kantorovich potential associated to the optimal transportation problem between two probability measures (although we do not need to require $\mu$ and $\nu$ have the same mass in our paper).

When $\mu$ and $\nu$ satisfy
\begin{equation}\label{eq:munul}
d\mu(x) \geq \frac 1 \lambda dx \quad  \mbox{ in } \Omega , \qquad d\nu(x) \leq \lambda dx\quad  \mbox{ in } \pa\psi(\Omega) 
\end{equation}
for some $\lambda>0$, then \eqref{eq:ineq1} implies that
$$ \frac 1{\lambda^2} |A| \leq  |\pa\psi(A)|  \quad \mbox{ for all Borel sets $A\subset \Omega$}$$
which says that $\psi$ is a solution (in the {\it Alexandrov sense}) of the one-sided Monge-Amp\`ere equation
\begin{equation}\label{eq:MA1} 
\det (D^2 \psi(x)) \geq \frac 1{\lambda^2}  \quad \mbox{ in } \Omega .
\end{equation}
In the case of absolutely continuous measures, inequality~\eqref{eq:ineq1} is hence a reformulation of sub-solutions to the Monge-Amp\`ere equation in terms of optimal transportation.

In dimension $2$, it is known that \eqref{eq:MA1}  implies that $\psi$ is strictly convex in $\Omega$ (see Theorem \ref{th:strictConvexity2D} below) and the strict convexity of $\psi$ is the first step in the regularity theory for the solutions of the  full Monge-Amp\`ere equation \eqref{eq:pureMongeAmpere} and the corresponding optimal transportation problem (see Section \ref{sec:ref} for further presentation of relevant results and references). In dimension $2$ and only in dimension $2$ (see below again), this is a {\em local property} in the sense that it does not require any boundary condition on $\psi$: Whenever $\psi$ satisfies \eqref{eq:MA1} on any subdomain $\Omega$, it is strictly convex in the interior of that subdomain, independently of what occurs outside of $\Omega$.  
 
The goal of this paper is to extend this result when $\psi$ satisfies \eqref{eq:ineq1} with measures $\mu$ and $\nu$ that are not necessarily absolutely continuous with respect to the Lebesgue measure and satisfy some lower and upper bounds only up to a certain scale (Assumption \ref{ass:measures} below).
This includes the case where $\mu$ and $\nu$ are sums of uniformly distributed Dirac masses.
Our main theorem  (Theorem \ref{th:flatpartbound}) implies in particular that $\psi$ is then strictly convex up to a certain scale (in Corollary \ref{cor:flat}, we derive a bound on the diameter of the flat parts of $\psi$).
Equivalently, our result says that the Legendre transform of $\psi$ is $C^1$ regular up to a certain scale (see Corollary \ref{cor:regularity}) on a subset of $\Omega$.

\medskip

{\bf Outline of the paper:} 
Our first result, Theorem \ref{thm:optimal}  in Section \ref{sec:optimal} below, makes precise the relation between inequality \eqref{eq:ineq1} and optimal transportation theory.  It is proved in Section \ref{sec:proof} and is of independent interest. The strict convexity "up to a certain scale", which is the main topic of this paper is then presented in  Section  \ref{sec:main} together with several consequences. This is followed by an overview of the existing literature and results related to our work (Section \ref{sec:ref}).
The proof of our main theorem is given in Section \ref{sec:proofthm}, while Sections   \ref{sec:proofcor1} and \ref{sec:proofcor2} are devoted to the proofs of Corollaries \ref{cor:flat} and \ref{cor:regularity}.

\medskip

\subsection{Relation to Optimal transportation problem}\label{sec:optimal}

Inequality \eqref{eq:ineq1} is natural in the context of optimal transportation theory (in any dimension $n\geq 1$). To explain this connection, we first recall that given two probability measures  $\mu \in \mathscr{P}(\RR^n)$, $\nu \in \mathscr{P}(\RR^n)$ the {\em Monge-Kantorovich Problem} with quadratic cost is concerned with the minimization problem
\begin{equation}
\label{eq:KantorovichQuad}
\min_{\pi\in \Pi(\mu,\nu)} \int_{ \RR^n \times  \RR^n} |x-y|^2  d\pi(x,y)
\end{equation}
where
$\Pi(\mu,\nu)$ denotes the set of all probability measures $\pi \in \mathscr{P}( \RR^n \times \RR^n)$ with marginals $\mu$ and $\nu$, i.e. such that $\pi(A \times  \RR^n)=\mu(A)$ for all $A \subset  \RR^n$ and $\pi( \RR^n \times B)=\nu(B)$ for all $B \subset   \RR^n$. 

The existence of a minimizer $\pi$ for problem (\ref{eq:KantorovichQuad}) (which will be called an \textit{optimal transport plan} between the measures $\mu$ and $\nu$) is a classical result, see for example \cite{villani2008optimal}. Moreover, it is known that $\pi\in \Pi(\mu,\nu)$ is a minimizer for \eqref{eq:KantorovichQuad} {\it if and only if} it is supported on the graph of the subdifferential of a lower semi-continuous convex  function $\psi$:
 \begin{equation}\label{eq:supgraph}
 \mathrm{supp}(\pi) \subset \mathrm{Graph}(\pa\psi):= \left\{(x,z)\in  \RR^n\times \RR^n\; |\; z\in \pa\psi(x)\right\} \mbox{.}
 \end{equation}
 The function $\psi$ is called a Kantorovich potential for problem (\ref{eq:KantorovichQuad}) and we can easily check that it satisfies \eqref{eq:ineq1} (globally). Indeed, for all Borel sets $A\subset \RR^n$, we can write
$$ 
\mu(A) = \int_A \int_{\RR^n} d\pi =  \int_A \int_{\pa \psi(A)} d\pi \leq  \int_{\RR^n} \int_{\pa \psi(A)} d\pi  = \nu(\pa\psi(A)).
$$
We note that this inequality might be strict in general but that we can similarly prove the inequality
\begin{equation}\label{eq:ineq2}
 \mu(\pa\psi^*(A)) \geq \nu(A)
\end{equation}
where $\psi^*$ denotes the Legendre transform defined by
$$\psi^*(z) = \sup_{x\in\RR^n} \big( x\cdot z -\psi(x)\big).$$
(indeed, the duality relation $ z\in\pa\psi(x) \Leftrightarrow x\in \pa\psi^*(z)$
together with \eqref{eq:supgraph} implies
$ \mathrm{supp}(\pi) \subset \left\{(x,z)\in  \RR^n\times \RR^n\; |\; x\in \pa\psi^*(z)\right\} .$)

\medskip

The purpose of this paper is to derive strict convexity estimates on $\psi$ when we assume only that  \eqref{eq:ineq1}  {\em holds locally} (and we do not require \eqref{eq:ineq2}). 
If we assume that both \eqref{eq:ineq1} and \eqref{eq:ineq2} hold (locally), then we get some convexity estimates on $\psi^*$, which are equivalent to $C^1$ estimates on $\psi$.

A question that arises naturally is whether assuming \eqref{eq:ineq1} and \eqref{eq:ineq2} is equivalent to assuming that $\psi$ is associated to an optimal transport plan between $\mu$ and $\nu$.
The answer is of course straightforward for the Monge-Amp\`ere equation since a  function that is a sub-solution and a super-solution is necessarily a solution. It is hence natural for \eqref{eq:ineq1} and \eqref{eq:ineq2} to imply a similar result, but the situation is more delicate. We recall in particular that the Monge-Kantorovich potential is not unique when the measures $\mu$ and $\nu$ are not absolutely continuous with respect to the Lebesgue Measure.

To the authors' knowledge, this question has not been previously addressed in the literature and we will 
show that if $\psi$ satisfies \eqref{eq:ineq1} and \eqref{eq:ineq2} globally (so in $\RR^n$) and if $\mu$ and $\nu$ have finite second moment, then $\psi$ must indeed be a Kantorovich potential associated to the optimal transportation problem (\ref{eq:KantorovichQuad}).
More precisely, we prove the following result:
\begin{theorem}\label{thm:optimal}
Let  $\mu,\;\nu$ be two probability measures on $\RR^n$ with finite second moment:
$$
\int_{\RR^n} |x|^2d\mu(x) + \int _{\RR^n} |y|^2  d\nu(y) <\infty
$$
and let $\psi$ be a proper lower semi-continuous  convex  function such that 
\eqref{eq:ineq1} and \eqref{eq:ineq2} hold for all measurable sets $A$. Then there exists an optimal plan $\pi\in \Pi(\mu,\nu)$ (minimizer of \eqref{eq:KantorovichQuad})
supported on $\Gamma = \mathrm{Graph}\,{\pa \psi}$ and the pair $(\psi,\psi^*)$ is then a minimizer for the dual problem
$$\inf \left\{ \int \psi d\mu +\int \vphi\, d\nu \, ;\, x\cdot y \leq \psi(x)+\vphi(y) \quad\forall (x,y)\right\}.$$
\end{theorem}
This result is completely independent from the regularity theory that we develop in the rest of this paper, but it clarifies the relation between equation \eqref{eq:ineq1} and the optimal transportation problem.
The proof (see Section \ref{sec:proof}) relies on the approximation of the measures $\mu$ and $\nu$ by sums of Dirac masses to construct a  plan $\pi\in \Pi(\mu,\nu)$ supported on $\Gamma = \mathrm{Graph}\,{\pa \psi}$ (the optimality of such a plan then follows from classical results from optimal transportation theory).

 
 \subsection{Local convexity and regularity}\label{sec:main}
 The main goal of this paper  is to develop a regularity theory when we do not assume that $\mu$ and $\nu$ are absolutely continuous with respect to the Lebesgue measure, but that they satisfy the following assumption:
\begin{assumption}
	\label{ass:measures}
	Assume that there are constants $h_1,h_2>0$ and $\lambda_1,\lambda_2>0$, such that the measures $\mu$ and $\nu$ satisfy 
	\begin{equation}
	\label{eq:MeasuresConditions}
	\mu(R) \geq \frac{|R|}{\lambda_1}   \mbox{\, and \,} \nu(R'\cap \pa\psi(\Omega)) \leq \lambda_2 |R'|
	\end{equation} 
	for any rectangles $R \subset  \Omega$, $R' \subset  \RR^2$ with dimensions at least $h_1$ and $h_2$ in every direction for $R$ and $R'$ respectively.

\end{assumption} 

This assumption is less restrictive than \eqref{eq:munul} and is
 relevant in the framework of optimal transportation. 
In fact, the original problem considered by Kantorovich in \cite{kantorovitch1958translocation}  included measures $\mu$ and $\nu$ that are sums of Dirac masses rather than absolutely continuous measures. 
This setting is also important for numerical applications:
Measures satisfying Assumption \ref{ass:measures} appear naturally when introducing discrete approximations of absolutely continuous measures with bounded densities, as is often done for computational purposes.

In order to state our main result, we introduce the set
$$
\Omega^\delta = \{ x\in\Omega\, ;\, \mathrm{dist}(x,\pa\Omega)>\delta\}.
$$
The main result of this paper is the derivation of the technical inequality \eqref{eq:flatpartbound} below, which quantify the strict convexity of $\psi$ up to a scale depending on $h_1$ and $h_2$:
\begin{theorem}[Strict convexity of Kantorovich potentials $\psi$]
	\label{th:flatpartbound}
	Let $\psi:\RR^2\to\RR$ be a convex function satisfying \eqref{eq:ineq1} for some measures $\mu$ and $\nu$ satisfying Assumption~\ref{ass:measures}.
	 Given $\delta>0$ and  $(x,y) \in \Omega^\delta \times \Omega^\delta $, let $K$ be any constant satisfying
	\begin{equation}\label{eq:defK}
	\opsi \geq \mathrm{diam} \, \partial \psi (U_{\delta})
	\end{equation}
	where $U_{\delta}$ is a $\delta$-neighborhood of the segment $[x,y]$ (i.e. $U_{\delta}=\displaystyle \cup_{z \in [x,y]} B_{\delta}(z)$) and define 
	$$ \oeps = - \min_{t\in [0,1]} \psi((1-t)x+ty) -[(1-t)\psi(x)+t\psi(y)] \geq 0 .$$
	There exists a universal constant $C$ such that if the length $\ell := |x-y|/2$ satisfies
	\begin{equation}\label{eq:ellcond}	
	\ell\geq 2\,h_1,\quad \ell^{2}\geq C\,K \,\lambda_1\,\lambda_2\,h_2\,,
	\end{equation}	
	then the following inequality holds:
	\begin{equation}
	\label{eq:flatpartbound}
	{\ell^8} \,  \log \left(1+\frac{\delta}{\gamma}\right)
	\leq C\,\lambda_1^4\,\lambda_2^4\,\opsi^8,
	\end{equation}
	provided 
	$$
	\gamma:=\max \left\{ \frac{\oeps }{\opsi},2h_1, \frac{\ell h_2}{C\,\opsi }\right\} \leq \delta/2.
	$$
\end{theorem}

We immediately make the following remarks:
\begin{enumerate}
\item The logarithm in the left hand side of \eqref{eq:flatpartbound} goes to infinity when $\gamma$ goes to zero. So Theorem~\ref{th:flatpartbound} provides a lower bound on $\gamma$ depending on the quantity $\frac{\ell^2}{C \lambda_1 \lambda_2 \opsi ^2}$. Indeed we have that either $\gamma > \delta / 2$ or  inequality (\ref{eq:flatpartbound}) provides a lower bound for $\gamma$.
So with the notations of the theorem, we see that as long as \eqref{eq:ellcond} holds, we have
	\begin{equation}
	\label{eq:flatpartbound2}
	\gamma  \geq 
	 \delta  \min\left\{  \frac{1}{\exp \left( \frac{C^4 \lambda_1^4 \lambda_2^4 K^8}{ \ell^8} \right) -1} , \frac{1}{2} \right\} \mbox{.}
	\end{equation}
Note that we can take  $K=  \mathrm {diam}\, \pa \psi (\Omega)$ which does not depends on $\ell$.
\item The conditions \eqref{eq:ellcond} are clearly satisfied in the absolutely continuous case $h_1=h_2=0$. In that case, \eqref{eq:flatpartbound2} gives a lower bound on $\gamma =\oeps/K$ and implies the strict convexity of $\psi$.  We actually recover a classical result, see Theorem \ref{th:strictConvexity2D} below.

\item The proof will make it clear that the assumption $(x,y) \in \Omega^\delta \times \Omega^\delta $ in the theorem is not necessary. The result holds for $(x,y) \in \Omega \times \Omega$ provided there is a rectangle $R_{\delta}(x,y)$, with base equal to the line segment $[x,y]$ and height equal to $\delta$ which  is contained in $\Omega$. In this setting we can also take $K=\mathrm{diam} (\partial \psi(R_{\delta}(x,y)))$.

\item As mentioned above, the conditions \eqref{eq:ellcond} are trivially satisfied when $h_1=h_2=0$. When $h_1,h_2\neq 0$, it is clear that we need some conditions on $\ell$ since we expect the potential $\psi$ to have flat parts and so $\oeps=0$ if $\ell$ is small enough.
In the simple case where $\mu$ and $\nu$ are uniformly distributed Dirac masses (on lattices of characteristic length $h_1$ and $h_2$),
 then the first condition in \eqref{eq:ellcond} is necessary to have several lattice points  in the set $U_{\delta}$, while the second condition in  \eqref{eq:ellcond} will guarantee that all those points cannot be sent onto a thin rectangle (of height $h_2$).

 \item The result is consistent with the natural scaling of the problem. For example,  
if we replace the measure $\nu$ by the new measure $\tilde \nu$ defined by
$ \tilde \nu (R) := \nu (\tau R)$ for some fixed $\tau>0$, then $\tilde \nu$ satisfies the conditions of 
Assumption \ref{ass:measures} with $\tilde h_2 = \tau^{-1} h_2$ and $\tilde \lambda_2 = \tau^2 \lambda_2$.
Furthermore, the function $\tilde \psi = \tau^{-1} \psi$ is a Kantorovich potential associated to the measures $\mu$ and $\tilde \nu$ which satisfies inequality \eqref{eq:ineq1} (with $\tilde \nu$ instead of $\nu$). One can then check that the conditions \eqref{eq:ellcond} and the inequality \eqref{eq:flatpartbound} are unchanged by these transformations.
\end{enumerate}

Theorem \ref{th:flatpartbound} provides a way to quantify how close $\psi$ is to being strictly convex. 
For instance, we can use Theorem \ref{th:flatpartbound}  to estimate the largest possible length of a "flat part" of $\psi$ by assuming that $\oeps=0$ and using \eqref{eq:flatpartbound} to get an upper bound on $\ell$.
We get:

\begin{corollary}\label{cor:flat}
Under the conditions of Theorem \ref{th:flatpartbound}, assume furthermore that 
$\eps =0$ (that is, $\psi$ is affine on the segment $[x,y]$).

If $h_1\leq \delta/4$, $\sqrt{C \lambda_1 \lambda_2} h_2 \leq \delta$ and 
\begin{equation}\label{eq:condellh2}
\frac{\ell h_2}{\opsi }\leq \frac{C \delta}{2}
\end{equation}
then 
\begin{equation}\label{eq:ellbound}
\ell \leq \max\left\{  2 h_1, \sqrt{ C \lambda_1 \lambda_2 \, \opsi h_2}, 
 \frac{\opsi \sqrt {C \lambda_1 \lambda_2}}{\left[\ln\left(\frac\delta{2 h_1}\right)\right]^{1/8}}, 
  \frac{\opsi \sqrt {C \lambda_1 \lambda_2}}{\left[ \ln \left(\frac{\delta}{\sqrt{C \lambda_1 \lambda_2} h_2}\right)\right]^{1/8}} 
\right\} \mbox{.}
\end{equation}
\end{corollary}
We recall that we can take $K= \mathrm {diam}\, \pa\psi(\Omega)$ in which case  \eqref{eq:condellh2} reads $ \ell h_2 \leq C \mathrm {diam}\, \pa\psi(\Omega)$ and  \eqref{eq:ellbound} gives
an upper bound on $\ell$ which only depends on the data of the problem and goes to zero when $\max\{ h_1,h_2\}\to 0$.
When $h_1=h_2=0$, Corollary \ref{cor:flat} gives $\ell=0$, 
so we recover the classical result that $\psi$ must be strictly convex in that case (no flat parts).

We can also take  $\opsi= \mathrm{diam} \, \partial \psi (U_{\delta})$ (so that \eqref{eq:ellbound} is sharper) 
in which case we note that if $h_2^2 \leq \frac{\delta \ell}{\lambda_1\lambda_2}$ then we can use the estimate \eqref{eq:opsilower}, derived further in the proof,  to replace the condition \eqref{eq:condellh2} with the following condition that does not depend on $\ell$:
\[
 \sqrt {C \lambda_1 \lambda_2} h_2 \leq \frac{3\delta^{3/2}}{(\mathrm {diam}\, \Omega_2)^{1/2}}.
\]

Going back to Theorem \ref{th:flatpartbound}, we observe that the control it provides on the convexity of $\psi$ should imply some $C^1$ regularity up to some length scale depending on $h_1,h_2$ on the Legendre dual or conjugate  (see \cite{rockafellar2009variational}) defined for all $z\in\RR^2$ by
\begin{equation}
\label{eq:LegendreTransform}
\psi^*(z) = \sup_{x\in\Omega} \big( x\cdot z -\psi(x)\big).
\end{equation}
Indeed, we can show:
\begin{corollary}[$C^1$ regularity of $\psi^*$]\label{cor:regularity}
	Under the conditions of Theorem \ref{th:flatpartbound} and given $\delta>0$,
	there exists some functions $\rho(\ell)$, $\rho_1(\ell)$ and $\rho_2(\ell)$ monotone increasing, with limit $0$ when  $\ell\to 0^+$, 
	and depending only on $\delta$, $\lambda_1\lambda_2$, $D=\mathrm{diam} \,\Omega$ and $K$ such that
	for all $z,z'\in \pa \psi(\Omega^\delta)$, we have 
	\begin{equation}\label{eq:reg}
	|x-x'| \leq \max \left(\rho(|z-z'|) ,   \rho_1(  h_1),\rho_2(h_2) ,  \right) \qquad \forall x\in \pa\psi^*(z),\; x'\in\pa\psi^*(z').
	\end{equation}

In particular if $h_1=h_2=0$ then $\psi^*$ is $C^1$ in $\pa \psi(\Omega^\delta)$ with the explicit estimate on the modulus of continuity of $\nabla \psi^*$,
\begin{equation}\label{eq:oepslowerbd}
  |\nabla{\psi^*}(z)-\nabla{\psi^*}(z')| \leq  C\,\frac{\sqrt{\lambda_1\,\lambda_2}\,L_\infty}{\left(\log\left(1+\frac{1}{ C\,\sqrt{\lambda_1\,\lambda_2}\,|z-z'|}\right)\right)^{1/8}}
  \end{equation}
where $L_\infty$ now denotes the Lipschitz bound of $\psi$ over $\Omega$.
%
\end{corollary}

We conclude this presentation of our result by observing that in Assumption \ref{ass:measures} we only require a lower bound on $\mu$ and an upper bound on $\nu$. Such bounds are all that we need to study the strict convexity of $\psi$. 
Opposite bounds would be required to prove the $C^1$ regularity of $\psi$  up to a certain length scale (recall that $\psi^*$ is associated to an optimal transportation problem in which the roles of $\mu$ and $\nu$ are inverted). 
More precisely, if we assume  that 
 $$\mu(R\cap\Omega) \leq \lambda_2 |R| \quad \mbox{  and } \quad  \nu(R')\geq \frac{1}{\lambda_1} |R'| $$ 
for $R\subset \RR^2$ and $R'\subset \Lambda $ (up to a certain scale), then using 
inequality \eqref{eq:ineq2} instead of \eqref{eq:ineq1} our analysis yields
 the $C^1$ regularity up to a certain scale for  the potential $\psi$ on the set
 $$ 
 \Omega' = \bigcup_{\delta>0} \pa\psi^* (\Lambda^\delta)
 $$
where $\Lambda^\delta = \{ x\in\Lambda\, ;\, \mathrm{dist}(x,\pa\Lambda)>\delta\}.$

\subsection{Brief overview of the literature}\label{sec:ref}

When the measures $\mu=f\,dx$ and  $\nu = g\, dx$
are absolutely continuous and concentrated on the open sets $\Omega$ and $\Lambda$,
 a classical result due to Brenier (\cite{brenier1987decomposition, brenier1991polar}) states that the solution of the minimization problem (\ref{eq:KantorovichQuad}) is unique and is given by $\pi=(Id\times \na\psi)_{\#} \mu$,  where $\psi:\RR^n\to \RR$ is a globally Lipschitz convex function  such that
  $\na\psi _{\#} \mu=\nu $.


If furthermore, 
there exists $\lambda>0$ such that 
$1/\lambda \leq f,g\leq \lambda$ on their respective supports, then $\psi$ satisfies
\begin{equation}
\label{eq:pureMongeAmpere}
 \frac{1}{\lambda^2} \chi_{\Omega}\leq \det D^2 \psi \leq \lambda^2 \chi_{\Omega}
\end{equation}
in a weak sense (the {\it Brenier sense}) together with the  boundary condition $\na\psi(\RR^n)\subset \Lambda $ (see for instance  \cite{caffarelli1992regularity, philippis2013regularity}).

Even in that case, it is classical that the regularity of $\psi$ requires some condition on the support of $g$ (for example if $\Omega$ is connected but $\Lambda$ is not, then the map $\na \psi$ must be discontinuous).
Caffarelli proved in \cite{caffarelli1992regularity} that if we further assume that 
$\Lambda $ is {\it convex}, then $\psi$ is a strictly convex  solution of the Monge-Amp\`ere equation \eqref{eq:pureMongeAmpere}  in the following {\it Alexandrov sense}:
\begin{equation}
\label{eq:MAalex} \frac{1}{\lambda^2} | A\cap \Omega| \leq |\pa\psi(A)| \leq \lambda^2 |A\cap\Omega| \qquad \mbox{for any Borel set $A\subset \RR^n$.}
\end{equation}
The regularity theory for Monge-Amp\`ere equation developed by Caffarelli \cite{caffarelli1990localization,caffarelli1991some,caffarelli1992boundary,caffarelli1992regularity,caffarelli1996boundary} for strictly convex solutions of \eqref{eq:MAalex} then  implies that $\psi$ is $C^{1,\alpha}_{loc}$.

Even in this absolutely continuous framework, our result requires only inequality \eqref{eq:ineq1} - which is equivalent to the lower bound in \eqref{eq:MAalex} - and {\em no assumption on $\Lambda$}. We note that this inequality is always satisfied by Brenier's potential, while the upper bound in  \eqref{eq:MAalex} requires further assumptions on $\Lambda$ (e.g. convexity) to hold.
When $\Lambda$ is non-convex, the convex potential  $\psi$ cannot be expected to be $C^1$ everywhere.
However, partial regularity results have been derived that offer a useful comparison, first in dimension $2$ by Yu \cite{Yu07} and Figalli \cite{Figalli10} and then generalized to higher dimension in \cite{FigalliKim10} and to more general cost functions in \cite{philippisfigalli15}.
These results show in particular that there exists an open subset of $\Omega$   of full measure in which $\psi$ is $C^{1,\alpha}$ and strictly convex. In dimension $2$, a precise geometric description of the singular set can be given \cite{Figalli10}. Our argument provides explicit quantitative estimates in that sense that extend to non absolutely continuous measures.




In this paper, we do not need to assume that $\psi$ is associated to an optimal transportation problem with nice properties globally. We only require inequality \eqref{eq:ineq1} to hold for measures $\mu$ and $\nu$ that satisfy some lower and upper bounds in some subsets of their supports.
In dimension $3$ and higher, functions that satisfy \eqref{eq:pureMongeAmpere} in $\Omega$ might not be strictly convex, as shown by  Pogorelov's classical counterexample \cite{gutierrez2001monge} and
the regularity theory for such solutions requires  appropriate assumptions on the boundary conditions \cite{caffarelli1990localization,caffarelli1991some,caffarelli1992boundary,caffarelli1996boundary}.
However, in dimension $2$ (which is the setting of our main result), 
there is a simple proof of the strict convexity of smooth local solutions of \eqref{eq:pureMongeAmpere}, which was originally proved in \cite{aleksandrov1942smoothness} and \cite{heinz1959differentialungleichung} by Aleksandrov and Heinz independently (see also \cite{trudinger2008monge}). 
That result can be formulated as follows:
\begin{theorem} [\cite{aleksandrov1942smoothness}, \cite{heinz1959differentialungleichung}]
	\label{th:strictConvexity2D}
	For $n=2$, let $ \psi \in C^2_{loc}(\Omega)$ satisfy 
\begin{equation}\label{eq:MAS} 
\det \, D^2\psi \geq \lambda^{-2} >0 \qquad\mbox{  in $\Omega$,}
\end{equation}
and assume that $\psi \geq 0$ in $\Omega$ and $\psi (x_0)=0$ for some $x_0$ in the interior of $\Omega$. Denote $\delta:=\mbox{dist} (x_0, \partial \Omega) >0$ and let $H$ be any line passing through $x_0$.
	Then for all $\ell \leq \frac{\delta}{2}$, the quantity
	$$ \gamma = \frac { \sup_{x \in B_{\ell}(x_0) \cap H} \psi(x)}{\| \nabla \psi \|_{L^\infty(\Omega)}}.
	$$
	satisfies 
	\begin{equation}\label{eq:estcontinuous}
	\ell^2 \ln\left(1+ \frac{\delta}{\gamma}\right) \leq 8\lambda^2 \| \nabla \psi \|_{L^\infty(\Omega)} ^2.
	\end{equation}
\end{theorem}
Inequality \eqref{eq:estcontinuous} implies the following estimate:
\begin{equation}
\label{eq:strictConvexity2D}
\sup_{x \in B_{\ell}(x_0) \cap H} \psi(x) \geq \frac{\| \nabla \psi \|_{\infty} \delta}{\exp \left(\displaystyle \frac{8 \lambda^2  \| \nabla \psi \|_{\infty}^2}{ \ell^2}\right)-1}>0
\end{equation}
for all $\ell \leq \frac{\delta}{2}$.

Our Theorem \ref{th:flatpartbound} with $h_1=h_2=0$ gives a proof of this result, and in particular estimate \eqref{eq:estcontinuous}, when we assume only that $\psi$ is an {\it Alexandrov} solution of \eqref{eq:MAS}, that is  
a convex function satisfying
$$ |\pa\psi(A)| \geq \frac{1}{\lambda^2} | A\cap \Omega| \quad \mbox{ for all Borel sets in $\Omega$.}$$
But the main interest of our result is that we consider measures that may not be absolutely continuous with respect to the Lebesgue measure. In this case, the Kantorovich potential $\psi$ (which still exists but may not be unique) does not solve the Monge-Amp\`ere equation (either \eqref{eq:pureMongeAmpere} or \eqref{eq:MAalex}). 
To our knowledge, no quantitative estimates on the convex function $\psi$ are known in this setting. Brenier's result does not apply (there might not be any measurable map $T$ such that  $T _{\#} \mu=\nu $), and Kantorovich potentials are not expected to be either convex (they will have `flat parts') nor $C^1$ (they will have `corners').

Theorem \ref{th:flatpartbound} is the derivation of an inequality similar  to \eqref{eq:estcontinuous} when \eqref{eq:MAS} is replaced by \eqref{eq:ineq1} with
 measures satisfying Assumption \ref{ass:measures} with $\lambda^2=\lambda_1\lambda_2$ and $\gamma$ replaced by
$$ \max \left\{\gamma ,2h_1, \frac{\ell h_2}{C\,\opsi }\right\} .$$
This implies in particular that (in dimension $2$) any Kantorovich potential $\psi$ is strictly convex up to some scale depending on $h_1$ and $h_2$ in any open set in which the lower and upper bounds \eqref{eq:MeasuresConditions} hold.

Although our result is similar to Theorem \ref{th:strictConvexity2D}, 
the proof is completely different since we cannot use the Monge-Amp\`ere equation \eqref{eq:MAS} in this non absolutely continuous setting and instead we must rely solely on the measure inequality \eqref{eq:ineq1}.
In the groundbreaking work of Caffarelli as well as in the partial regularity theory of
 \cite{Yu07,Figalli10,FigalliKim10}, a key tool is the use of some variants of the maximum principle for the Monge-Amp\`ere equation and the use of appropriate barriers. This is not possible in our framework. 
Instead, the proof of our Theorem \ref{th:flatpartbound} relies on the derivation of upper and lower bounds for an integral quantity defined in \eqref{eq:quantityToBound}-\eqref{eq:weight}.
Note that  a variational approach (relying on optimal transportation arguments rather than using some barriers for Monge-Amp\`ere equation) to the partial regularity theory of \cite{FigalliKim10}  was recently developed in \cite{Goldman17, Goldman18}.

\medskip

\medskip

It is natural to ask whether our result could be extended to  dimension $n\geq 3$. 
It turns out that even in the absolutely continuous case, the result of Theorem \ref{th:strictConvexity2D} does not hold in dimension $3$ and higher. Indeed, a classical example by Pogorelov shows that $\psi$ can have a flat part
and is thus not necessarily strictly convex
 (see \cite{gutierrez2001monge}). 
A natural extension of Theorem \ref{th:strictConvexity2D} can however be found in \cite[Theorem 2.34]{bonsante2017equivariant} : Under conditions similar to Theorem \ref{th:strictConvexity2D}  but in dimension $n\geq 3$,
 the convex function $\psi$ cannot be affine on a set of dimension larger than or equal to $n/2$.
For the sake of completeness, we present in Appendix \ref{app:2} a short proof, based on the ideas of  \cite{bonsante2017equivariant},
of the following quantitative estimate
\begin{theorem}
	\label{th:affinebound}
	Let $n\geq 3$ and let $ \psi \in C^2$, $\psi \geq 0$ satisfy $ \det D^2\psi \geq \lambda^{-2}>0$ and assume that $\psi(x_0)=0$ with $\delta:=\mbox{dist} (x_0, \partial \Omega) >0$. Let $H$ be an affine surface of dimension $d$ passing through $x_0$, then for all $\ell \leq \frac{\delta}{2}$, the quantity
	$$ \gamma = \frac {\sup_{x \in B_{\ell}^n(x_0) \cap H} \psi(x) } {\| \nabla \psi \|_{L^\infty(\Omega_1)}} 
$$
satisfies 
	\begin{equation}\label{eq:estgammand} 
\ell^{2d}  \varphi(\delta/\gamma)  \leq C\lambda^2  \| \nabla \psi \|_{L^\infty(\Omega_1)}^n \delta^{2d-n}
\end{equation}
with $\varphi(s) := s^{2d-n}\int_0^{s}   \frac{r^{n-d-1}}{( r +1 )^d}\, dr$.
	\end{theorem}
We note that $\varphi$ satisfies  $\lim_{s\to\infty} \varphi(s) = \infty$ if and only if $d\geq n/2$ and so \eqref{eq:estgammand} implies the following lower bound:
\begin{equation}
	\label{eq:affinebound}
	\sup_{x \in B_{\ell}^n(x_0) \cap H} \psi(x) \geq
	\left\{
	\begin{array}{ll}
	\min\left\{ \delta  \|\nabla \psi \|_{\infty}, \left( \frac{\ell^{2d}}{C\lambda^2  \|\nabla \psi \|_{\infty}^{2n-2d}} \right)^{\frac{1}{2d-n}} \right\}
	  & \mbox{ if $d > n/2$} \\[3pt]
	\delta \|\nabla \psi \|_{\infty}  \exp \left( -C \frac{\lambda^2 \|\nabla \psi \|_{\infty}^n}{ \ell^n} \right) & \mbox{ if $d = n/2$}.
	\end{array}
	\right.
	\end{equation}

In view of this result, it seems that the main result of this paper (Theorem \ref{th:flatpartbound}) could be extended to higher dimensions, provided one considers hypersurfaces of dimension $n/2$.
However, the  basic tool of our proof, the integral quantity  \eqref{eq:quantityToBound}-\eqref{eq:weight}, is not well suited for such a generalization, and a new quantity would need to be introduced.
This question will thus be addressed in a future work.


\section{Proof of Theorem \ref{th:flatpartbound}}\label{sec:proofthm}
\subsection{Preliminaries}
The Kantorovich problem with the quadratic cost function  is invariant under rigid motions.
Up to a translation and a rotation of $\Omega$, we can thus assume that the points $x,y$ in Theorem \ref{th:flatpartbound} 
are $a:=(-\ell,0)$ and $b:=(\ell,0)$ and that the rectangle $[-\ell,\ell] \times [0,\;\delta]$ is contained in $\Omega$.

Up to subtracting an affine function, we can also assume that $\psi$ satisfies 
\begin{equation}\label{eq:convexfunc}
\psi(-\ell,0)=\psi(\ell,0)=0 \quad \mbox{ and } \quad 0\in\pa\psi([a,b]).
\end{equation}

Throughout the proofs,  $x=(\xp,\xo)$ or $y=(y_\pl,y_\pe)$ will denote  points in $\Omega\subset \RR^2$ with $\xp,\;y_\pl$ the first coordinate parallel to the segment $[a,b]$.  Similarly $z=(z_\pl,z_\pe)$ will denote a point in $\pa\psi(\Omega)\subset \RR^2$.

We will also use the following notation:
 \begin{align}
 \regdelta & =\{(\xp,\,\xo)\,|\;|\xp|\leq \ell/2,\ 0 \leq \xo\leq \delta\}.   
  \end{align}
Furthermore,
\eqref{eq:convexfunc} implies that $0\in \partial \psi (U_{\delta})$ and so for any $
\opsi \geq \mathrm{diam} \, \partial \psi (U_{\delta})$ we have
\begin{equation}\label{eq:opsi}
\opsi\geq \| \partial \psi \|_{L^{\infty}( \Rd) }= \sup_{y\in\Rd} \sup_{z\in \pa\psi(y)}|z| .
\end{equation}
We also note that
\begin{equation}\label{eq:oeps} 
\oeps := - \min_{t \in [0,1]} \psi(t a + (1-t) b)\geq 0.
\end{equation}

Throughout the proofs,  $C$ denotes a numerical constant, which depends only on the dimension $d=2$ and whose value may change from line to line in the calculations.

Before moving to the heart of the proof, we state the following simple lemma which we will use repeatedly,
\begin{lemma}
	\label{lemma:SubdifferentialBound}
	Let $\psi:[-\ell,\ell] \times [0, 2 \delta] \rightarrow \mathbb{R}$ be a convex function satisfying \eqref{eq:convexfunc} and \eqref{eq:oeps}.
	Then for all $y\in \Omega$ such that $|y_\pl|\leq \ell/2$ we have
	\[
	|z_\pl| \leq \frac{2}{\ell}\, \left(\opsi\, |y_\pe|+\oeps \right), \qquad \forall z\in \pa\psi(y).
	\]
\end{lemma}

\begin{proof}
  Consider any $y\in \Omega$ with $|y_\pl|\leq \ell/2$, $0 \leq y_\pe \leq 2\,\delta$ and any $z\in\partial\psi(y)$. 
  Then  we have by the definition of subdifferential
	\[
	\psi(b) \geq \psi(y) +z\cdot(b-y)=\psi(y)+z_\pl\cdot (b_\pl-y_\pl) +z_\pe\cdot (b_\pe-y_\pe).
	\]
 Since $b_\pl - y_\pl = \ell - y_\pl \geq \ell /2 $, and $a_\pe = 0 $, this lets us deduce that:
	\begin{align*}
	|z_\pl| 
	& \leq \frac{1}{b_\pl - y_\pl}\, \left[z_\pe\cdot (y_\pe-a_\pe) +(\psi(b)-\psi(y))\right]\\
	& \leq \frac{2}{\ell}\, \left[z_\pe\cdot y_\pe +(\psi(b)-\psi(y))\right]\\
	&  \leq \frac{2}{\ell}\left(\opsi\, |y_\pe|+\oeps \right), 
	\end{align*}
	where we have used \eqref{eq:convexfunc} so $\psi(b)= 0$, \eqref{eq:oeps} so $\psi(y)\geq -\oeps$ and the fact that $|z_\pe|\leq \opsi $ (by \eqref{eq:opsi}). 
This completes the proof of Lemma \ref{lemma:SubdifferentialBound}.
\end{proof}

We conclude these preliminaries by noting that the quantity $ \mathrm{diam} \, \partial \psi (U_{\delta})$   {\em a priori} depends on $\ell$. We obviously have
\begin{equation}\label{eq:opsiupper} 
\mathrm{diam} \, \partial \psi (U_{\delta}) \leq  \mathrm {diam}\, \pa\psi(\Omega)
\end{equation}
and we can show the following lower bound:
\begin{lemma}\label{lem:opsibound}
If $h_1 \leq \min\left\{ \delta , \ell\right\}$ and 
$h_2^2 < \frac{\delta \ell}{\lambda_1\lambda_2}$, then 
\begin{equation}\label{eq:opsilower} 
\mathrm{diam} \, \partial \psi (U_{\delta}) \geq \left( \frac{\delta \ell}{\lambda_1\lambda_2}\right)^{1/2}.
\end{equation}
\end{lemma}

\begin{proof}
Inequality \eqref{eq:ineq1} gives
		\begin{align*}
		\mu(U_\delta) \leq \nu(\partial \psi (U_\delta))  .
		\end{align*}
Since the dimensions of $U_\delta$ satisfy $ \min \{\delta, \ell\}\geq h_1$,	
Assumption  \ref{ass:measures}  implies
$$		
\mu(U_\delta)\geq \frac{\ell \delta}{\lambda_1} ,\quad  \mbox{ and } \quad \nu(\partial \psi (U_\delta)) \leq \lambda_2 \max \{(\mathrm{diam} \, \partial \psi (U_{\delta}))^2,h_2^2\}.
$$		
We deduce
$$ \frac{\ell \delta}{\lambda_1\lambda_2} \leq \max \{(\mathrm{diam} \, \partial \psi (U_{\delta})) ^2,h_2^2\}$$
and the  condition $h_2^2 < \frac{\delta \ell}{\lambda_1 \lambda_2}$ implies  (\ref{eq:opsilower}).
\end{proof}

\subsection{Proof of Theorem \ref{th:flatpartbound}}
We now describe our strategy for proving Theorem \ref{th:flatpartbound}.
First, we note that since $\psi$ is a convex function in $\Omega$, it is differentiable in a subset $\widetilde\Omega\subset \Omega$ of full measure ($|\Omega\setminus \widetilde\Omega|=0$), see for instance~\cite{rockafellar1970convex}.

We can thus define a map $T:\Omega\mapsto \RR^2$  which satisfies 
\begin{equation}\label{eq:T}
T(x):=\nabla \psi(x)\qquad  \forall x\in \widetilde\Omega.
\end{equation}

Our proof of Theorem \ref{th:flatpartbound} relies on some careful estimates (upper and lower bounds) of the integral quantity
\begin{equation}
\label{eq:quantityToBound}
\int_{\regdelta\times\regdoubledelta} |T_\pe(y)-T_\pe(x)|\, \varphi(x,y) \, dy \, dx 
\end{equation}	
where the 
weight function $\varphi(x,y)$ is given by
\begin{equation}
\label{eq:weight}
\varphi(x,y)=\frac{1}{(x_\pe+\gamma)^{2}} \mathbf{1}_{\{\frac{1}{2} x_\pe\leq y_\pe \leq 2\, x_\pe\}},
\end{equation}
for some $\gamma>0$. 
 The exponent $2$ is chosen to obtain the right logarithmic divergence in the estimates. 

Using the notations from Theorem \ref{th:flatpartbound}, we will first prove the following upper bound which does not require \eqref{eq:ineq1}:
\begin{proposition}
	\label{prop:UpperBound}
Assume that $\psi:[-\ell,\ell]\times[0,2\delta]\rightarrow \mathbb{R}$ is a convex function satisfying \eqref{eq:convexfunc}.
	Then there exists a universal constant $C>0$ s.t. the following inequality holds for all $\gamma\geq \oeps/K$
	\begin{equation}\label{eq:upperbd}
\int_{\regdelta\times\regdoubledelta} |T_\pe(y)-T_\pe(x)|\, \varphi(x,y) \, dy \, dx 
\leq C\, \opsi\, \ell^{2}\,\left( \left[\log\left(1+\frac{\delta}{\gamma}\right)\right]^{1/2} + 1 \right),
    	\end{equation}
	where we recall that  $\opsi $ and $\oeps $   satisfy  \eqref{eq:opsi} and \eqref{eq:oeps}.
\end{proposition}
The proof of this upper bound is fairly straightforward (see Section \ref{sec:upper}) and only makes use of the convexity of $\psi$ and  Lemma \ref{lemma:SubdifferentialBound}.

Next, we will prove the following lower bound for \eqref{eq:quantityToBound}:
\begin{proposition}
\label{prop:LowerBound}	
Let $\psi$ be a convex function satisfying \eqref{eq:ineq1} for some measure $\mu$ and $\nu$ satisfying Assumption~\ref{ass:measures}.
Assume further than $\psi$ satisfies \eqref{eq:convexfunc}. There exists a universal constant $C$ s.t. assuming that $\ell$ satisfies \eqref{eq:ellcond}, which we recall is
\[
\ell\geq 2\,h_1,\quad \ell^{2}\geq C\,\opsi\,\lambda_1\,\lambda_2\,h_2,
\]
and defining 
\begin{equation}\label{eq:gamma}
\gamma:=\max \left( \frac{\oeps }{\opsi},\;2\,h_1, \;\frac{\ell h_2}{C\,\opsi }\right),
\end{equation}
then the following inequality holds
	\begin{equation}\label{eq:LowerBound}\begin{split}
            &\int_{\regdelta\times\regdoubledelta} |T_\pe(y)-T_\pe(x)|\, \varphi (x,y)\, dy\, dx \\
            &\qquad \qquad\geq
\frac{\ell^{4}}{C\,\lambda_1\,\lambda_2\,\opsi }\,\left(1\wedge\frac{\ell^{2}}{\lambda_1\,\lambda_2\,\opsi^2 }\right)\,\log \left(\frac{1}{2}+\frac{\delta}{2\,\gamma}\right),\\
\end{split}
	  \end{equation}
provided $\gamma <\delta$ and where we recall the notation $a\wedge b=\min(a,\,b)$.
\end{proposition}
The proof of this proposition, which is presented in Section \ref{sec:lower}, is more delicate. 
This is where we use the fact that $\psi$ satisfy the Monge-Amp\`ere like condition \eqref{eq:ineq1} with measures $\mu$ and $\nu$ satisfying~\eqref{eq:MeasuresConditions}.

\medskip

\begin{proof}[Proof of Theorem \ref{th:flatpartbound}]
The key to conclude the proof of Theorem \ref{th:flatpartbound} is that the bounds provided by Propositions \ref{prop:UpperBound} and \ref{prop:LowerBound} scale differently in $\ell$ and $\gamma$. Combining the two will hence naturally lead either to an upper bound on $\ell$ or to a lower bound on $\gamma$. More precisely we directly obtain from \eqref{eq:upperbd}  and \eqref{eq:LowerBound} that
\[
\frac{\ell^{2}}{\lambda_1\,\lambda_2\,\opsi^{2}}\,\left(1\wedge\frac{\ell^{2}}{\lambda_1\,\lambda_2\,\opsi^{2}}\right)\,\log \left(\frac12+\frac{\delta}{2\,\gamma}\right)\leq C\, \left( \left[\log \left(1+\frac{\delta}{\gamma}\right) \right]^{1/2} + 1 \right).
\]
Since we assumed in the theorem that $\delta\geq 2\,\gamma$, we have $\log \left(1+\frac{\delta}{\gamma}\right)\leq C\,\log \left(\frac12+\frac{\delta}{2\,\gamma}\right)$ so that we can simplify the inequality above to:
\begin{equation}\label{eq:ghgj}
\frac{\ell^{2}}{\lambda_1\,\lambda_2\,\opsi^{2}\, }\,\left(1\wedge\frac{\ell^{2}}{\lambda_1\,\lambda_2\,\opsi^2}\right)\,\left[\log \left(1+\frac{\delta}{\gamma}\right) \right]^{1/2}\leq C \left( 1 + \left[\log \left(1+\frac{\delta}{\gamma}\right) \right]^{-1/2}  \right) \leq C.
\end{equation}
Moreover we also get $\log \left(1+\frac{\delta}{\gamma}\right)\geq \log 3$ (still using the assumption that $\delta\geq 2\,\gamma$) so \eqref{eq:ghgj} gives
\[
\frac{\ell^{2}}{\lambda_1\,\lambda_2\,\opsi^{2}} \left( 1 \wedge \frac{\ell^2}{\lambda_1 \lambda_2 K^2} \right) \leq C [\log 3]^{-1/2},
\]
which can be used to show that
\[
\left( \frac{\ell^2}{\lambda_1 \lambda_2 K^2} \wedge \left(\frac{\ell^2}{\lambda_1 \lambda_2 K^2} \right)^2  \right) \geq C  \left(\frac{\ell^2}{\lambda_1 \lambda_2 K^2} \right)^2.
\]
for some (different) constant $C$.
Together with \eqref{eq:ghgj}, this finally implies
$$
\left(\frac{\ell^{2}}{\lambda_1\,\lambda_2\,\opsi^2}\right)^2\,\left[\log \left(1+\frac{\delta}{\gamma}\right) \right]^{1/2}\leq C 
$$
which completes the proof of Theorem \ref{th:flatpartbound}.
\end{proof}

\subsection{Upper bound: Proof of Proposition \ref{prop:UpperBound}}\label{sec:upper}
\begin{proof}[Proof of Proposition \ref{prop:UpperBound}]
We first assume that $\psi$ is $C^2$ so that all the computations below make sense.
We can write	
\begin{align}
 & \int_{\regdelta\times\regdoubledelta} |T_\pe(y)-T_\pe(x)|\, \varphi(x,y) \, dy\, dx  \nonumber \\
&\qquad\qquad = \int_{\regdelta\times \regdoubledelta}  \left|\int_{0}^{1} \nabla T_\pe (x+t(y-x))\cdot (y-x) \,dt\,  \right| \varphi(x,y)\, dy\, dx\nonumber \\
&\qquad\qquad \leq  \int_{\regdelta\times\regdoubledelta}  \int_{0}^{1} \left|\partial_\pl T_\pe (x+t(y-x))\cdot (y_\pl-x_\pl) \right|\, dt \,   \varphi(x,y)\,dy\,dx\nonumber \\
  &\qquad\qquad \quad  + \int_{\regdelta\times \regdoubledelta} \int_{0}^{1} \left|\partial_\pe T_\pe(x+t(y-x)) \cdot (y_\pe-x_\pe) \right|\, dt \,  \varphi(x,y)\,dy\,dx,\nonumber
\end{align}
where $\partial_\pl$ denotes the derivative with respect to the first component and $\partial_\pe$ is the derivative in the orthogonal direction. Using the symmetry of the expression in $x$ and $y$, we have
\begin{align}
 & \int_{\regdelta\times\regdoubledelta} |T_\pe(y)-T_\pe(x)| \varphi(x,y) \, dy\, dx  \nonumber \\
&\leq  2\,\int_{\regdelta\times\regdoubledelta}  \int_{1/2}^{1} \left|\partial_\pl T_\pe (x+t(y-x))\cdot(y_\pl-x_\pl) \right|\, dt \,   \varphi(x,y)\,dy\,dx\nonumber \\
&\qquad\qquad \quad  + 2\,\int_{\regdelta\times\regdoubledelta} \int_{1/2}^{1} \left|\partial_\pe T_\pe(x+t(y-x))\cdot (y_\pe-x_\pe) \right|\, dt \,  \varphi(x,y)\,dy\,dx,\label{eq:TT}
\end{align}
To bound the first term in  the right-hand side, 
we note that by definition of $R_\delta$, $|y_\pl-x_\pl|\leq \ell$ so that
using the change of variable $y\to z=x+t(y-x)$
\begin{align}
& \int_{\regdelta\times\regdoubledelta}  \int_{1/2}^{1} \left|\partial_\pl T_\pe (x+t(y-x))\cdot (y_\pl-x_\pl) \right|\,   \varphi(x,y)\, dt  \,dy\,dx\nonumber \\
&\qquad\qquad   \leq  \ell\, \int_{\regdelta}  \int_{1/2}^{1}   \int_{\regdoubledelta} \left|\partial_\pl T_\pe (x+t(y-x)) \right| \,    \varphi(x,y)\,dy\, dt\,dx \nonumber\\
&\qquad\qquad   \leq  \ell\, \int_{\regdelta}  \int_{1/2}^{1}   \int_{\regdoubledelta} \left|\partial_\pl T_\pe (z) \right|\,  \mathbf{1}_{x+\frac{z-x}{t}\in \Rd}   \varphi\left(x,x+\frac{z-x}{t}\right)\frac{1}{t^d}\,dz\, dt\,dx \nonumber \\
& \qquad\qquad  \leq  \ell \int_{\regdoubledelta} \left|\partial_\pl T_\pe (z)  \right| J_1(z) \,dz.\label{eq:term1}
\end{align}
Using the definition of $\varphi(x,y)$ (see \eqref{eq:weight})
and the notation
\[
\Omega_{x_\pe} =\left\{ y \in [\ell/2,\ell/2] \times [0, \delta] ;\;   \frac{1}{2}\, x_\pe \leq y_\pe \leq 2\,x_\pe  \right\},
\]
we get that the weight $J_1(z)$ is equal to
\begin{align}
 J_1(z) & = 2\,\int_{1/2}^{1}   \int_{\regdelta}     {\bf 1}_{x+\frac{z-x}{t}\in \Rd}   \varphi\left(x,x+\frac{z-x}{t}\right)\,dx \, dt\nonumber\\
& =
2\, \int_{1/2}^{1}    \int_{\regdelta}  \frac{1}{(|x_\pe|+\gamma)^{2}}   {\bf 1}_{x+\frac{z-x}{t}\in  \Omega_{x_\pe} }     \,dx \, dt  .
\nonumber 
\end{align}

 Observe that the definition of $\Omega_{x_\pe}$ is actually symmetric on $R_\delta$: $y\in \Omega_{x_\pe}$ iff $x\in \Omega_{y_\pe}$ since $x_\pe \geq 0$.  Consequently  $z\in \Omega_{x_\pe}$ implies that $x\in \Omega_{z_\pe}$ as $x\in R_\delta$ and 
\[\begin{split}
J_1(z)&\leq \frac{C}{(|z_\pe|+\gamma)^{2}}\,\int_{1/2}^{1}    \int_{\regdelta} {\bf 1}_{x+\frac{z-x}{t}\in  \Omega_{x_\pe} }     \,dx \, dt\leq \frac{C}{(|z_\pe|+\gamma)^{2}}\,\int_{\Omega_{z_\pe}}\,dx\\
&\leq \frac{C}{(|z_\pe|+\gamma)^{2}}\,\ell\,|z_\pe|\leq \frac{C\,\ell}{(|z_\pe|+\gamma)} . 
\end{split}
\]
Going back to  \eqref{eq:term1}, we find
\begin{equation}\label{eq:term11}
\int_{\regdelta\times\regdoubledelta}  \int_{1/2}^{1} \left|\partial_\pl T_\pe (x+t(y-x))\cdot(y_\pl-x_\pl) \right|\,   \varphi(x,y)dt  \,dy\,dx \leq C\, \ell^2\, \int_{\regdoubledelta}\frac{ \left|\partial_\pl T_\pe (z)  \right| }{(|z_\pe|+\gamma)} \,dz.
\end{equation}
Next, we note that the convexity of $\psi$ implies that the matrix
\[
\left[\begin{matrix}
&\partial_\pl T_\pl &\partial_\pe T_\pl\\
&\partial_\pl T_\pe &\partial_\pe T_\pe\\
  \end{matrix}\ \right]
\]
is symmetric and non-negative with a non-negative determinant:
$\partial_\pl T_\pl (z) \partial_\pe T_\pe(z) - \partial_\pl T_\pe \partial_\pe T_\pl \geq 0$, which implies that $|\partial_\pl T_\pe(z)| \leq |\partial_\pl T_\pl(z)|^{1/2}\,|\partial_\pe T_\pe(z)|^{1/2}$.
This lets us deduce that
		\begin{align}
 \int_{\regdoubledelta} \frac{ |\partial_\pl T_\pe(z)|}{|z_d|+\gamma} dz &  \leq  \int_{\regdoubledelta}  \frac{|\partial_\pl T_\pl |^{1/2}}{(|z_\pe|+\gamma)} |\partial_\pe T_\pe(z)|^{1/2} \, dz\nonumber  \\
		&  \leq  \left[\int_{\regdoubledelta } \frac{|\partial_\pl T_\pl(z)|}{(|z_\pe|+\gamma)^2} dz \right]^{1/2} \left[\int_{\regdoubledelta }  |\partial_\pe T_\pe(z)|\, dz\right]^{1/2}\label{nabla'Td} .
				\end{align}
 Using the fact that $\partial_\pl T_\pl \geq 0$ from the convexity of $\psi$,
 \begin{align}
 \int_{\regdoubledelta} \frac{| \partial_\pl T_\pl (z) |}{(z_\pe + \gamma)^2} d z_\pe & = \int_{0}^{ \delta} \frac{T_\pl (\ell/2, z_\pe) - T_\pl (-\ell/2, z_\pe)}{(z_\pe + \gamma)^2} d z_\pe \leq \frac{C}{\ell} \int_{0}^{\delta} \frac{(K z_\pe + \oeps)}{(z_\pe + \gamma)^2} d z_\pe \\
 & \leq \frac{C K}{\ell} \int_{0}^{\delta} \frac{1}{(z_\pe +\gamma)} d z_\pe = C K \ell^{-1}  \int_{0}^{\delta} \frac{1}{z_\pe + \gamma} dz_\pe,
 \label{partialjTj}
 \end{align}
 by using Lemma \ref{lemma:SubdifferentialBound} and the fact that $\gamma\geq \oeps/K$.
 
                Similarly, we have that $\partial_\pe T_\pe \geq 0$ so that
                \begin{equation}
\int_{\regdoubledelta} |\partial_\pe T_\pe (z)| dz \leq \int_{-\ell/2}^{\ell / 2} [T_\pe (z_\pl, \delta)- T_\pe (z_\pl, 0)] d z_\pl \leq 2 K \ell.
\label{partialdTd}
\end{equation}
Combining \eqref{partialjTj} and \eqref{partialdTd} into \eqref{nabla'Td} and inserting the result into \eqref{eq:term11}, we conclude that
\begin{equation}\label{eq:term111}
\int_{\regdelta\times\regdoubledelta}  \int_{1/2}^{1} \left|\partial_\pl T_\pe (x+t(y-x))(y_1-x_1) \right|   \varphi(x,y)dt  \,dy\,dx \leq C\,  \ell^{2}\, \opsi \,\left[ \int_{0}^{2 \delta} \frac{1}{z_\pe + \gamma} d z_\pe \right]^{1/2}. 
\end{equation}
which gives a bound for the first term in the right hand side of \eqref{eq:TT}.

\medskip

We now proceed similarly to bound the second term in the right-hand side of \eqref{eq:TT}. First we write, recalling that $ \partial_\pe T_\pe\geq 0$,
\begin{align*}
&  \int_{\regdelta\times\regdoubledelta} \int_{1/2}^{1} \left|  \partial_\pe T_\pe(x+t(y-x))\,(y_\pe-x_\pe) \right|\,dt\,   \varphi(x,y)\,dy\,dx\\
&\qquad\qquad   \leq   \int_{\regdelta}  \int_{1/2}^{1}   \int_{\regdelta} \partial_\pe T_\pe (x+t(y-x))\, |y_\pe-x_\pe|\,      \varphi(x,y)\,dy\, dt\,dx. \nonumber
\end{align*}
Note that the definition of $\varphi$ in \eqref{eq:weight} implies that
\begin{align*} 
|y_\pe-x_\pe|\,  \varphi(x,y) & =\frac{ |y_\pe -x_\pe| }{(x_\pe+\gamma)^2} \mathbf{1}_{\{\frac{1}{2} x_\pe\leq y_\pe \leq 2 x_\pe \}} \\
& \leq \frac{1}{x_\pe +\gamma } \mathbf{1}_{\{\frac{1}{2} x_\pe \leq y_\pe \leq 2 x_\pe \}}.
\end{align*} 
We perform the same change of variable $z=x+t(y-x)$ as we used after \eqref{eq:term1}) to find that
\begin{align}
&  \int_{\regdelta\times\regdoubledelta} \int_{1/2}^{1} \left|  \partial_\pe T_\pe (x+t(y-x))\,(y_\pe -x_\pe ) \right|\, dt\,   \varphi(x,y)\,dy\,dx\nonumber \\
&\qquad\qquad   \leq  \int_{\regdelta}  \int_{1/2}^{1}   \int_{\regdoubledelta} \partial_\pe T_\pe (z)\,   \frac{1}{x_\pe +\gamma }\,   \mathbf{1}_{x+\frac{z-x}{t}\in \Omega_{x_\pe}}    \,dz\, dt\,dx \nonumber \\
& \qquad\qquad  \leq   \int_{\regdoubledelta} \partial_\pe T_\pe (z)\,   J_2(z) \,dz,\label{eq:term2}
\end{align}
with 
\[
J_2(z) = 2\,\int_{1/2}^{1}  \int_{\regdelta}   \frac{1}{x_\pe +\gamma }   \mathbf{1}_{x+\frac{z-x}{t}\in \Omega_{x_\pe }}    \, dx\, dt.
\]
Proceeding as with the weight $J_1(z)$ above (the only difference lies in the power of $(x_\pe +\gamma)$), we find that $\mathbf{1}_{x+\frac{z-x}{t}\in \Omega_{x_\pe}}\leq\mathbf{1}_{x\in \Omega_{z_\pe}}$
\[
J_2(z)\leq \frac{C}{z_\pe +\gamma}\,|\Omega_{z_\pe}|\leq C\, \ell.
\]
Inserting this bound in \eqref{eq:term2}, we obtain
\begin{align}
&   \int_{\regdelta\times\regdoubledelta} \int_{1/2}^{1} \left|  \partial_\pe T_\pe(x+t(y-x))\,(y_\pe-x_\pe) \right|\,dt \,  \varphi(x,y)\,dy\,dx\nonumber \\
& \qquad  \leq C\, \ell
  \int_{\regdoubledelta} \partial_\pe T_\pe (z)\,dz
= C\, \ell
 \int_{-\ell / 2}^{\ell / 2} (T_\pe(z_\pl,\delta)-T_\pe(z_\pl,0))\,  dz_\pl \label{eq:secondtermreg}
\leq C\, \opsi\,  \ell^{2}.
 \end{align}
Combining \eqref{eq:secondtermreg} and \eqref{eq:term111} in \eqref{eq:TT}, we obtain that
\begin{equation}
\int_{R_{\delta}\times R_{2 \delta}} |T_\pe(x)-T_\pe(y)|\,\varphi(x,y)\,dy\,dx\leq C\, \opsi\, \ell^{2}\, \left( \left[\log\left(1+\frac{\delta}{\gamma}\right)\right]^{1/2} + 1 \right), \label{conclusionprop1}
  \end{equation}
which proves the proposition if $T$ is $C^1$ and hence $\psi$ is $C^2$.
\medskip

When $\psi$ is only convex but not $C^2$, we naturally introduce the convex function $\psi^\eta= \psi \star_x \rho_{\eta}$, where $\rho_{\eta}$ is a standard mollifier. We may then apply \eqref{conclusionprop1} to $\psi^\eta$ and find for $T^\eta=\nabla\psi^\eta$
\[
\int_{R_{2 \delta}\times R_{2 \delta}} |T_\pe^\eta(x)-T_\pe^\eta(y)|\,\varphi(x,y)\,dy\,dx\leq C\, \opsi_\eta\, \ell^{2}\, \left( \left[\log\left(1+\frac{\delta}{\gamma}\right)\right]^{1/2} + 1\right),
\]
where we observe that, in this case, since we only integrate over $R_\delta$, $\opsi_\eta$ is given by
\[
\opsi_\eta=\sup_{R_\delta} |\nabla\psi_\eta|\leq \|\partial \psi\|_{L^\infty(\regdoubledelta)}\leq\opsi,
\]
for $\eta<\delta$. At the same time, since $\psi$ is convex then $T=\nabla\psi$ belongs to $BV(\regdoubledelta)$ and therefore $\|T^\eta-T\|_{L^1(\regdoubledelta)}\to 0$ as $\eta\to 0$. Since $\varphi$ is bounded for any fixed $\gamma>0$, we may directly pass to the limit $\eta\to 0$ and obtain \eqref{conclusionprop1} on $T$. 
\end{proof}
\subsection{Lower bound: Proof of Proposition \ref{prop:LowerBound}\label{sec:lower}}
We now turn to the proof of the lower bound \eqref{eq:LowerBound}.   Given $x_\pe\in (0, \delta)$, we recall for convenience the definition of the set $\Omega_{x_\pe}$, the following sets
\[
\Omega_{x_\pe} =\left\{ y \in [-\ell / 2, \ell / 2] \times [0, \delta]\, ;\; \  \frac{1}{2}\, x_\pe \leq y_\pe \leq 2\, x_\pe   \right\},
\]
together with the more restricted set 
\[
\Lambda_{x_\pe} =\left\{ y \in [-\ell / 4, \ell / 4] \times [0, \delta]\, ;\; \   x_\pe \leq y_\pe \leq \frac32\, x_\pe   \right\}.
\]
Since we are trying to show that $T_\pe(y)$ cannot be concentrated, instead of looking at $|T_\pe(y)-T_\pe(x)|$, we define, for $\xi \in \RR$ and $\eta>0$, the more general set
\begin{equation}
\label{eq:subdifftoolarge}
\Omega_{x_\pe,\eta}=\left\{ y \in \Omega_{x_\pe} \,;\, |z_\pe-\xi| > \eta \mbox{ for all } z \in \partial \psi(y) \right\} .
\end{equation}
Our first task, in Lemma \ref{lemma:etaExistence} below, is to show that for an appropriate value of $\eta$ and
for all $\xi\in \RR$, the set $\Omega_{x_\pe,\eta}$ is non empty, and more precisely $\Lambda_{x_\pe}\cap \Omega_{x_\pe,\eta}\neq\emptyset$. This will allow us to construct a half-cone within $\Omega_{x_\pe,\eta}$ in Lemma \ref{lemma:subdiffangle} and finally to obtain a lower bound for $|\Omega_{x_\pe,\eta}|$ in  Lemma \ref{lemma:SetLowerBound}.
This will finally let us  conclude the proof of Prop. \ref{prop:LowerBound} and obtain the lower bound \eqref{eq:LowerBound}.
%
%
\subsubsection{Non-emptyness of the  set \texorpdfstring{$\Omega_{x_\pe,\eta}$}{}}
First we have the following lemma, which implies in particular that the set $\Omega_{x_\pe,\eta}(\xi)$ is not empty.
\begin{lemma}
	\label{lemma:etaExistence}
	Let $\psi$ be a convex function satisfying \eqref{eq:ineq1} for some measure $\mu$ and $\nu$ satisfying Assumption~\ref{ass:measures} and let $K$ satisfy \eqref{eq:defK}. There exists a universal constant $C$  such that defining
	\begin{equation}\label{eq:etadef}
	  \eta := \frac{1}{C\, \lambda_1\, \lambda_2} \frac{   \ell^{2}}{\opsi },
	\end{equation}
	and assuming furthermore that  $\ell$ satisfies
	\[
        \ell \geq 2\, h_1,\quad  \ell^{2} \geq C\, \opsi\,   \lambda_1\, \lambda_2\, h_2,
        \]
 then for all $x_\pe > \gamma=\max ( \frac{\oeps }{\opsi},\,2h_1,\, \frac{\ell h_2}{C\,\opsi })$
and for all $\xi\in \RR$, there is at least one point $y^* \in \Lambda_{x_\pe}$ such that for some $z \in \partial \psi(y^*)$ we have $|z_\pe-\xi| \geq 3 \eta$.
\end{lemma}
The idea of the proof is to look at the image of the set $\Lambda_{x_\pe}$ by the subdifferential $\pa\psi$.
By Lemma~\ref{lemma:SubdifferentialBound} this image 
is bounded in the horizontal (i.e. $z_\pl$) directions.
However, \eqref{eq:ineq1} together with Assumption \ref{ass:measures} gives a lower bound on the measure of this image, which is where the fact that $\psi$ is the Kantorovich potential for an optimal transportation problem is crucial. 
Therefore the image cannot be too small in the vertical (i.e. $z_\pe$) directions, which is essentially the statement of Lemma \ref{lemma:etaExistence}.
The lower bounds on $\ell$ and $x_\pe$ in Lemma \ref{lemma:etaExistence} are necessary so that we can use  \eqref{eq:MeasuresConditions} on the measures $\mu$ and $\nu$.

\begin{proof}[Proof of Lemma \ref{lemma:etaExistence}]
We start by noticing that for all $z\in \partial \psi (\Lambda_{x_\pe})$, Lemma~\ref{lemma:SubdifferentialBound} implies that
\[
|z_\pl| \leq \frac{2}{\ell}\, \left(\opsi\, \frac{3}{2}\,x_\pe +\oeps \right) .
\]
This leads to defining the rectangle
\[
	\tilde R= \left\{z \in \RR^2\, ;\, |z_\pl| \leq  \max\left( \frac{2\,\opsi  }{\ell}\, \left(\frac{3}{2}\,x_\pe+\frac{\oeps}{\opsi } \right),h_2\right) \mbox{ and } |z_\pe-\xi| \leq \max(3\, \eta, h_2)  \right\}.
\]
	We show the existence of $y^*$ as in  Lemma \ref{lemma:etaExistence} by contradiction: Suppose there is no such point $y^*$, then we must have that $\partial \psi (\Lambda_{x_\pe}) \subset \tilde R$ and inequality \eqref{eq:ineq1} gives
	\begin{align}
	\mu(\Lambda_{x_\pe}) \leq \nu(\pa\psi(\Lambda_{x_\pe}))  \leq \nu(\tilde R\cap\pa\psi(\Omega) ).\label{eq:transportInequality}
	\end{align}
We now want to use Assumption \ref{ass:measures} to estimate the left and right hand side of \eqref{eq:transportInequality}.
The rectangle $\Lambda_{x_\pe}$ has size $\left(\frac \ell 2\right)\times\frac {x_\pe} 2$.
Since $\ell\geq 2\,h_1$ and $x_\pe \geq \gamma \geq 2h_1$, the rectangle $\Lambda_{x_\pe}$ has size at least $h_1$ in all directions and Assumption \ref{ass:measures} (see \eqref{eq:MeasuresConditions}) implies that $\mu(\Lambda_{x_\pe})\geq |\Lambda _{x_\pe} | /\lambda_1$.

Similarly, the definition of the set $\tilde R$ guarantees that $\tilde R$ has size at least $h_2$ in all directions and so
$\nu(\tilde R\cap\pa\psi(\Omega)) \leq \lambda_2 |\tilde R|$. Equation~(\ref{eq:transportInequality}) thus yields
\begin{equation}\label{eq:LambdaB}
|\Lambda_{x_\pe}| \leq \lambda_1\, \lambda_2\, |\tilde R|.
\end{equation}  
We now note that
\begin{equation}
|\Lambda_{x_\pe}| =C\,\ell \, x_\pe ,\label{LambdaB}
\end{equation}
while the assumption $x_\pe > \gamma$ with $\gamma\geq \frac{\oeps}{\opsi }$ and $\gamma\geq \ell\,h_2/C\,K$ implies
\begin{equation}
  \begin{split}
  |\tilde R| & =  C \,\max\left( \frac{2\opsi  }{\ell} \left(\frac{3}{2}\, x_\pe +\frac{\oeps}{\opsi } \right),\;h_2\right)\, \left(\max (3\, \eta, h_2)\right)\\
  &\leq  C\, \left(\frac{ \opsi }{\ell}\, x_\pe \right)\, \left(\max (3\, \eta, h_2)\right).
  \end{split}
  \label{intermed|B|}
\end{equation}

Together with the definition of $\eta$ this shows that
\begin{equation}
|\tilde R| \leq C \left(\frac{\opsi}{\ell} x_\pe\right) \max \left(\frac{1}{C \lambda_1 \lambda_2} \frac{\ell^2}{K}, h_2\right) 
\end{equation}

Equation (\ref{eq:LambdaB}) then proves that
\begin{equation}
C \ell x_\pe \leq \frac{\ell x_\pe}{C}
\end{equation} 
which is a contraction by taking $C$ large enough and concludes the second part of the proof.
\end{proof}
%
\subsubsection{Lower Bound on \texorpdfstring{$|\Omega_{x_\pe,\eta}|$}{}}
%
We now show that the measure $|\Omega_{x_\pe,\eta}|$ is bounded from below.
 We first need, as an intermediary result, the following lemma which only relies on the convexity of the function $\psi$. This lemma mostly states that if the subdifferentials corresponding to  two points $y'$ and $y''$ is concentrated in the vertical direction and the segment $[y',\;y'']$ is almost vertical, then the subdifferential corresponding to any point in that segment also has to be concentrated.
 
 We will later use this lemma together with Lemma \ref{lemma:etaExistence} to obtain contradictions and ensures the absence of concentration in the subdifferential over half a cone. 
 \begin{lemma}	\label{lemma:subdiffangle}
 	Let $\psi$ satisfy \eqref{eq:convexfunc}, consider any $x_\pe\geq \gamma=\max\left(\frac{\oeps}{\opsi},\,2\,h_1,\, \frac{\ell\,h_2}{C\,K}\right)$ and fix any $\xi \in \mathbb{R}$. 
 	Assume that $y',\;y'' \in \Omega_{x_\pe}$ are such that  $y'\neq\ y''$ and
 	\[
 	\partial \psi (y') \cap \{ z \in \Omega_2\, ;\,  |z_\pe-\xi| \leq \eta \} \neq \emptyset,\qquad \partial \psi (y'') \cap \{ z  \in \Omega_2\, ;\,  |z_\pe-\xi| \leq \eta \} \neq \emptyset.
 	\]
 	There exists a universal constant $C$ s.t., if 
 	\[
 	|\tan ((y',\;y''),\;e_\pe)| \leq \frac{\ell\, \eta}{C\, \opsi  x_\pe},
 	\]
 	where $(y',\,y''),\;e_\pe)$ is the angle between the vertical direction $e_\pe$ and the segment $[y',\,y'']$, 
 	then for all $y=s\,y' +(1-s)\, y''$ with $s \in (0,1)$ we have 
 	\[
 	\partial \psi (y) \subset \{ z \,;\,  |z_\pe-\xi| \leq 2\, \eta \}.
 	\]
 \end{lemma}
 
 \begin{proof}
 	Take $z' \in \partial \psi(y') \cap \{ |z_\pe-\xi| \leq \eta \}$ and $y=s\, y' +(1-s)\, y''$ for some fixed $s \in (0,1)$. We can assume (without loss of generality) that $y'_\pe-y_\pe >0$ and $y''_\pe-y_\pe<0$.
 	For any  $z\in \partial \psi (y)$, the convexity of $\psi$ implies (cyclical monotonicity of the sub-differential):
 	\[
 	(z'-z)\cdot  (y'-y) \geq 0, 
 	\]
 	and therefore
 	\[
 	(z'_\pe-z_\pe)\,(y'_\pe-y_\pe) \geq -(z'_\pl-z_\pl)(y'_\pl-y_\pl) \mbox{.}
 	\]
 	We hence deduce that
 	\[
 	z_\pe \leq z'_\pe +\frac{(z'_\pl-z_\pl)(y'_\pl-y_\pl) }{y'_\pe-y_\pe} \leq \xi+\eta +(|z'_\pl|+|{z_\pl}|)\,\frac{|y'_\pl-{y}_\pl| }{y'_\pe-y_\pe}\mbox{,}
 	\]
 	since $|\overline{z}_\pe-\xi|\leq\eta$. 
 	Using now Lemma \ref{lemma:SubdifferentialBound}, we then get that
 	\begin{align*}
 	z_\pe & \leq \xi+\eta + \left[\frac{2}{\ell}\, \left(\opsi\, |y'_\pe|+\oeps \right)+\frac{2}{\ell}\,  \left(\opsi\,  |y_\pe|+\oeps \right)\right]\,	\frac{|y'-y|}{y'_\pe-y_\pe}\\
 	& \leq \xi+\eta+\frac{C}{\ell}\,\opsi\, x_\pe \, |\tan ((y'',y'),\,e_\pe)|\leq \xi +2\eta,
 	\end{align*}
 	by the definition of the tangent and where we used the fact that $x_\pe\geq \gamma\geq {\oeps}/\opsi$, that $y',\, y'' \in \Omega_{x_\pe}$ so $y\in \Omega_{x_\pe}$ and as a consequence $y'_\pe,\;y_\pe\leq 2\,x_\pe$.
 	
 	Proceeding similarly using $y''$ instead of $y'$, we can get the inequality $z_\pe \geq \xi - 2\, \eta$ and the result follows.
 \end{proof}
 
 Using Lemmas \ref{lemma:etaExistence} and \ref{lemma:subdiffangle}, we can now get a lower bound on the measure of the set  $\Omega_{x_\pe,\eta}(\xi)$ (which we recall is defined by \eqref{eq:subdifftoolarge}). This will be the key estimate  in the proof of Proposition \ref{prop:LowerBound}.
 \begin{lemma}
 	\label{lemma:SetLowerBound}
 	Let $\psi$ be a convex function satisfying \eqref{eq:ineq1} for some measure $\mu$ and $\nu$ satisfying Assumption~\ref{ass:measures}. Assume further that $\psi$ satisfies  \eqref{eq:convexfunc}. Recall that $K$ satisfies \eqref{eq:defK} and that $\eta$ is defined by \eqref{eq:etadef}.
 	Assume furthermore that  $\ell $ satisfies, for an appropriate universal constant $C$,
 	\[
 	\ell \geq 2 h_1,\quad  \ell^{2} \geq C\,\opsi\,   \lambda_1\, \lambda_2\, h_2.
 	\]
 	Then, for all $x_\pe > \gamma=\max \left( \frac{\oeps }{\opsi},\;2\,h_1,\; \frac{\ell h_2}{C\,\opsi }\right)$,
 	and for all $\xi\in \RR$, one has the lower bound
 	\begin{equation}
 	\label{eq:OmegaLowerBound}
 	|\Omega_{x_\pe, \eta}(\xi)|\geq  \frac{\ell\,x_\pe}{C}\,\min\left(1,\ \frac{\ell^{2}}{\lambda_1\,\lambda_2\,\opsi^{2}}\right).
 	\end{equation}
 \end{lemma}
 
 \begin{proof}
 	Start by using Lemma \ref{lemma:etaExistence} to obtain the existence of one $\tilde y \in \Lambda_{x_\pe}$ be such that for some $\tilde z \in \partial \psi(y)$ we have $|\tilde z_\pe-\xi|\geq 3\, \eta$.
 	Define now $C_{\theta}$ as the cone (see figure \ref{fig:cones}) with vertex $\tilde y$ and  angle $\theta$ with the vertical direction $e_\pe$, such that
 	\[
 	\tan \theta=\min\left(\displaystyle \frac{\ell}{2 x_\pe},\,\displaystyle \frac{\ell\, \eta}{C\, \opsi  x_\pe}\right).
 	\]
 	Define furthermore the truncated cone $S_\theta=\{y\in C_\theta\,|\; |y_\pe-\tilde y_\pe|\leq x_\pe/2\}$.
 	
 	 	\begin{figure}[h]
 	 		\includegraphics[width=1\textwidth]{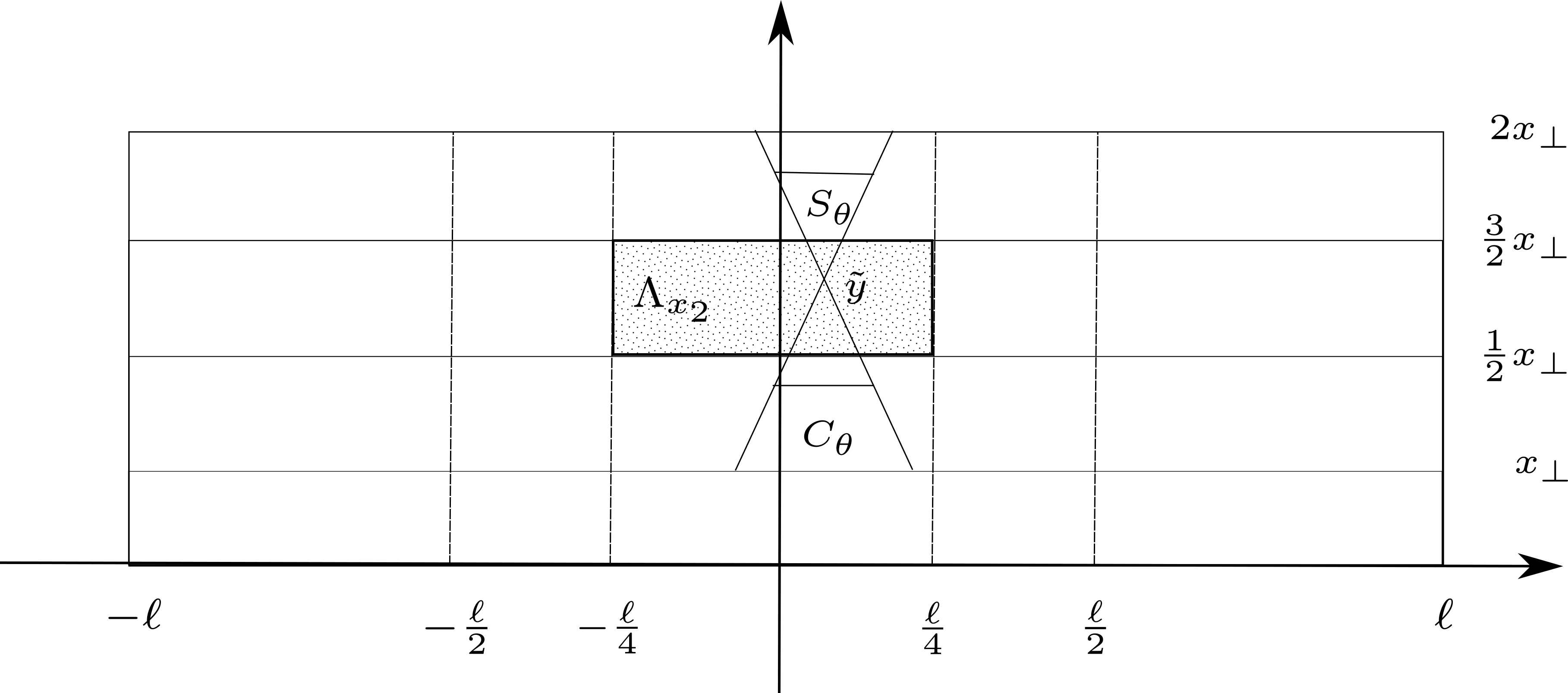}
 	 		\caption{Cones $C_{\theta}$ and $S_{\theta}$}
 	 		\label{fig:cones}
 	 	\end{figure}
 	
 	We first observe that $S_\theta\subset \Omega_{x_\pe}$ as for any $y\in S_\theta$, we have first $\tilde y_\pe-x_\pe/2\leq y_\pe\leq \tilde y_\pe+x_\pe/2$ and hence $\frac{x_\pe}{2}\leq y_\pe\leq 2\,x_\pe$ since $x_\pe\leq \tilde y_\pe\leq \frac{3}{2}\,x_\pe$. Second, since $|\tilde y_\pl|\leq \ell/4$, we have that
 	\[
 	|y_\pl|\leq |\tilde y_\pl|+|\tan \theta|\,|y_\pe-\tilde y_\pe|\leq \frac{\ell}{4}+|\tan \theta|\,\frac{x_\pe}{2}\leq \frac{\ell}{2},
 	\]
 	which is the reason for the condition $\tan \theta\leq \ell/(2\,x_\pe)$.
 	
 	The next step is to use Lemma \ref{lemma:subdiffangle} to prove that $|\Omega_{x_\pe,\eta}\cap S_{\theta}|\geq |S_\theta|/2$. For this consider any segment in the truncated cone $S_\theta$ with hence angle $\theta'\leq \theta$ with $e_\pe$. Denote by $L_{\theta'}^1$ and $L_{\theta'}^2$ the two half-parts of the segment from $\tilde y$.
 	
 	We can show that either $L_{\theta'}^1\subset \Omega_{x_\pe,\eta}$ or $L_{\theta'}^2\subset \Omega_{x_\pe,\eta}$. Indeed by contradiction, if this was not the case we would have some  $y' \in L^1_{\theta'}\setminus \Omega_{x_\pe,\eta}$, $y'' \in L_{\theta'}^2\setminus \Omega_{x_\pe,\eta}$. By the definition of $\Omega_{x_\pe,\eta}$, there exists $z' \in \partial \psi(y')$ with $|z'_\pe-\xi|\leq \eta$ and similarly $z''\in \partial \psi(y'')$ with $|z''_\pe-\xi|\leq \eta$.
 	
 	Of course by definition
 	\[  \tan(({y'},\,y''),\;e_\pe)=\tan\theta'\leq \tan\theta\leq \frac{\ell\,\eta}{C\,\opsi\,|x_\pe|},
 	\]
 	and we can directly apply Lemma \ref{lemma:subdiffangle}. As $\tilde y$ is a convex combination of $y',\,y''$, this implies that $\partial\psi(\tilde y)\subset\{z\,|\;|z_\pe-\xi|\leq 2\,\eta\}$ contradicting the fact that $\tilde z\in \partial\psi(\tilde y)$ but $|\tilde z_\pe-\xi|\geq 3\,\eta$.
 	
 	This proves as claimed that either $L_{\theta'}^1\subset \Omega_{x_\pe,\eta}$ or $L_{\theta'}^2\subset \Omega_{x_\pe,\eta}$ and integrating over all possible segments with all possible angles that $|\Omega_{x_\pe,\eta}\cap S_{\theta}|\geq |S_\theta|/2$.

 	To conclude the proof, it is hence enough to bound from below $|S_\theta|$,
 	\[
 	|S_\theta|=C \,x_\pe\,(x_\pe\,\tan\theta)\geq \frac{1}{C}\,x_\pe\,\ell\,\left(\min\left(1,\,\frac{\eta}{\opsi}\right)\right).
 	\]
 	Using the definition of $\eta$ in \eqref{eq:etadef}, we eventually obtain that
 	\[
 	|S_\theta|\geq \frac{\ell\,x_\pe}{C}\,\min\left(1,\ \frac{\ell^{2}}{\lambda_1\,\lambda_2\,\opsi^{2}\, }\right).
 	\]
 \end{proof}

\subsubsection{Proof of Proposition \ref{prop:LowerBound}}
We now have all the estimates required to prove Proposition \ref{prop:LowerBound}:

\begin{proof}[Proof of Proposition \ref{prop:LowerBound}]
Using the definition of $\varphi$ as given in \eqref{eq:weight}, and the set $\Omega_{x_\pe,\eta}(\xi)$ introduced in \eqref{eq:subdifftoolarge},  we get
			\begin{align*}
\int_{\regdelta\times\regdoubledelta} |T_\pe(y)-T_\pe(x)|\, \varphi (x,y)\, dy\, dx 
& = \int_{\regdelta} \frac{1}{(x_\pe+\gamma)^{2}}\, \int_{\Omega_{x_\pe}}|T_\pe(y)-T_\pe(x)| \, dy \, dx.\\
\end{align*}
Now fix $\xi=T_\pe(x)$ and calculate
\[
\int_{\Omega_{x_\pe}}|T_\pe(y)-T_\pe(x)|\,dy\geq \eta\,\int_{\Omega_{x_\pe,\eta}} dy=\eta\,|\Omega_{x_\pe,\eta}|.
\]
Observe that the assumptions on $\ell$ and the definition of $\gamma$ in Proposition \ref{prop:LowerBound} exactly coincide with Lemma~\ref{lemma:SetLowerBound}. Hence we may apply the lemma whenever $x_\pe>\gamma$ to find
\[
\begin{split}
  &\int_{\Omega_{x_\pe}}|T_\pe(y)-T_\pe(x)|\,dy\geq
\frac{\ell\,x_\pe\,\eta}{C}\,\min\left(1, \ \frac{\ell^{2}}{\lambda_1\,\lambda_2\,\opsi^{2}\, }\right)\\
&\qquad =\frac{\ell^{3}\,x_\pe}{C\,\lambda_1\,\lambda_2\,\opsi }\,\min\left(1,\ \frac{\ell^{2}}{\lambda_1\,\lambda_2\,\opsi^{2}}\right),\\
\end{split}
\]
by the definition of $\eta$.
This leads to
\begin{align*}
&\int_{\regdelta\times\regdoubledelta} |T_\pe(y)-T_\pe(x)|\, \varphi (x,y)\, dy\, dx 
\geq \int_{\regdelta\cap\{x_\pe \geq \gamma\}} |T_\pe(y)-T_\pe(x)|\, \varphi (x,y)\, dy\, dx\\
&\qquad\geq \frac{\ell^{3}}{C\,\lambda_1\,\lambda_2\,\opsi }\,\min\left(1, \ \frac{\ell^{2}}{\lambda_1\,\lambda_2\,\opsi^{2}\, }\right)\,\int_{\regdelta\cap\{ x_\pe \geq \gamma\}} \frac{dx_\pe}{(x_\pe+\gamma)}\\
&\qquad\geq \frac{\ell^{4}}{C\,\lambda_1\,\lambda_2\,\opsi }\,\min\left(1, \ \frac{\ell^{2}}{\lambda_1\,\lambda_2\,\opsi^{2}}\right)\,\int_\gamma^\delta \frac{dr}{r+\gamma},
\end{align*}
and we may conclude that
\[\begin{split}
&\int_{\regdelta\times\regdoubledelta} |T_\pe(y)-T_\pe(x)|\, \varphi (x,y)\, dy\, dx 
\\
&\quad\geq \frac{\ell^{4}}{C\,\lambda_1\,\lambda_2\,\opsi }\,\min\left(1,\ \frac{\ell^{2}}{\lambda_1\,\lambda_2\,\opsi^{2}}\right)\,\log \left(\frac{1}{2}+\frac{\delta}{2\,\gamma}\right),
\end{split}
\]
which completes the proof.
\end{proof}

\section{Proof of Corollary \ref{cor:flat}}\label{sec:proofcor1}

\begin{proof}[Proof of Corollary \ref{cor:flat}]
We now have that $\oeps=0$ and so
$$\gamma=\max \left\{ 2h_1,\frac{\ell h_2}{C\opsi }\right\}.$$

If the length $\ell$ does not satisfy \eqref{eq:ellcond} then we have	
\begin{equation}\label{eq:ellnotcond}	
\mbox{either }\quad  \ell < 2 h_1,\quad  \mbox{ or } \quad \ell^2 < C   \lambda_1 \lambda_2 \, \opsi h_2,
\end{equation}	
which gives the first two terms in \eqref{eq:ellbound}.

If $\ell$ satisfies \eqref{eq:ellcond}, then we note that 
since $h_1 \leq \delta/4$, condition \eqref{eq:condellh2} implies
$ \gamma < \delta/2$
and so we can use Theorem \ref{th:flatpartbound} to find
\[
 \left( \frac{ \ell^2}{C \lambda_1 \lambda_2 \opsi^2}\right)^4    \ln \left(\frac{\gamma + \delta}{  \gamma}\right)
 \leq  1
\]
Setting $u:=\frac{\ell}{ \opsi\sqrt {C \lambda_1 \lambda_2}}$, we rewrite this inequality as
\begin{equation}\label{eq:ineqflat1}
u^8   \ln \left(\frac{\gamma(u) + \delta}{\gamma(u)}\right)
 \leq  1,
\end{equation}
where $\gamma(u)= \max \left\{ 2h_1,\sqrt { \frac {\lambda_1 \lambda_2} C} h_2 u \right\}$.

When $\gamma(u)=2h_1\leq \delta/2 $, then \eqref{eq:ineqflat1} implies 
\begin{equation}\label{eq:u1} 
u \leq  \left[\ln\left(\frac\delta{2 h_1}\right)\right]^{-1/8}.
\end{equation}

When $\gamma(u)=  \sqrt { \frac {\lambda_1 \lambda_2} C} h_2 u $,
the assumption $ \frac{\sqrt{C \lambda_1 \lambda_2} h_2}{\delta} \leq 1$ implies $\gamma(u) \leq \frac{\delta u } {C}$ and so 
 the  inequality \eqref{eq:ineqflat1} gives
 $$
u^8   \ln \left(1+ \frac{C}{u}\right)\leq u^8   \ln \left(1+ \frac{\delta}{\gamma(u)}\right)\leq 1.
$$
Thus we obtain $u\leq C$
(since $C\geq1$ and $u\mapsto u^8   \ln \left(1+ \frac{C}{u}\right)\leq u^8$ is increasing). 
It follows that $\gamma(u)=  \sqrt { \frac {\lambda_1 \lambda_2} C} h_2 u \leq \sqrt{C\lambda_1\lambda_2} h_2$
and Inequality \eqref{eq:ineqflat1} then yields:
\begin{equation}\label{eq:u2}
 u^8\leq \left[ \ln \left(1+ \frac{\delta}{\gamma(u)}\right)\right]^{-1}\leq  
 \left[ \ln \left(\frac{\delta}{\sqrt{C \lambda_1 \lambda_2} h_2}\right)\right]^{-1}.
 \end{equation}

Inequalities \eqref{eq:u1} and \eqref{eq:u2} gives the last two terms in \eqref{eq:ellbound}
and conclude the proof of this first corollary.
\end{proof}

\section{Proof of corollary \ref{cor:regularity}}\label{sec:proofcor2}

Before proving Corollary \ref{cor:regularity}, we state the following lemma which is proved at the end of this section.
\begin{lemma}
	\label{lemma:sub}
	Let $\psi$ be a convex function on $\Omega$ and let $x,x'\in \Omega \times\Omega $. Denote $\ell=|x-x'|$ 
	and  
	\begin{equation}\label{eq:oeps2}
	\oeps = - \min_{t\in [0,1]} \psi((1-t)x+tx') -[(1-t)\psi(x)+t\psi(x')].
	\end{equation}
	Then, for any  $z \in \partial \psi(x)$ and $z' \in \partial \psi (x')$, we have that 
	\begin{equation}
	\label{eq:dualitybound}
	|z-z'| \geq 2 \frac{\oeps }{\ell} \mbox{.}
	\end{equation}	
\end{lemma}

\begin{proof}[Proof of Corollary \ref{cor:regularity}]
	We recall that
	$$ x\in \pa\psi^*(z) \Longleftrightarrow z\in\pa\psi(x)$$
	so we want to use \eqref{eq:flatpartbound2} to prove \eqref{eq:reg}. But in order to apply 
	\eqref{eq:flatpartbound2}, we first need to prove that the conditions of Theorem \ref{th:flatpartbound} are satisfied.
	We will first prove that
	\begin{equation}\label{eq:first}
	\pa\psi^*(z) \subset \Omega^{\delta/2} \quad\forall z\in \pa \psi(\Omega^\delta).
	\end{equation}
	This is not obvious, since  the definition of $\pa \psi(\Omega^\delta)$ only guarantees that there exists at least one 
	$\bar x\in \Omega^\delta$ such that $z\in\pa\psi(\bar x)$ (in other words
	$ \pa\psi^*(z) \cap \Omega^{\delta/2}\neq \emptyset$).
	We will prove \eqref{eq:first} by contradiction: Assume that there exists also $x\in \Omega\setminus \Omega^{\delta/2} $ such that 
	$z\in\pa\psi(x)$. Since $\psi$ is convex, this implies that $\psi$ must have a flat part along the segment $[\bar x,x]$. Indeed, the definition of the subdifferential implies that
	$$ \psi(tx+(1-t)\bar x) \geq \psi(\bar x) +t z\cdot (x-\bar x) \qquad \forall t\in[0,1]$$
	and
	$$ \psi(tx+(1-t)\bar x) \geq \psi(x) +(1-t) z\cdot (\bar x-x)\qquad \forall t\in[0,1]$$
	and a linear combination of those inequality yields
	$$
	\psi(tx+(1-t)\bar x) \geq (1-t)\psi(\bar x) + t\psi(x)\qquad \forall t\in[0,1].
	$$
	The convexity of $\psi $ implies that we must have equality in this inequality.
	
	After possibly replacing $x$ with the point $[\bar x,x] \cap \pa\Omega^{\delta/2}$, we deduce (since $\bar x\in \Omega^\delta$) that $\psi$ has a flat part of size at least $\delta/2$ in the set $\Omega^{\delta/2}.$
	By Corollary \ref{cor:flat}, this is impossible if 
	$h_1\leq k_1(\delta)$ and $h_2\leq k_2(\delta)$ for some functions $k_1$, $k_2$ depending only on $\lambda_1\lambda_2$ and $L_\infty$.
	This proves that \eqref{eq:first} must hold.
	\medskip

	Next, we  use Theorem \ref{th:flatpartbound}. We denote $\ell = |x-x'|$ and 
	assume that $h_1\leq \delta$ and that $\ell$ satisfies:
	\begin{equation}\label{eq:cond3}
	\ell \geq 2 h_1,\quad   \ell \geq \max\left\{ \sqrt{C  \lambda_1 \lambda_2 \, K h_2 } , \frac{\lambda_1\lambda_2}{\delta} h_2\right\}
	\end{equation}
	Then $h_1$ and $h_2$ satisfy
	\begin{equation}\label{eq:h1h2}
	h_1\leq \min\left\{ \delta , \ell/2 \right\} \quad h_2 \leq \min\left\{ \sqrt{\frac{\delta \ell}{\lambda_1\lambda_2}}, \frac{\ell^2}{C\lambda_1\lambda_2 \, K }\right\} .
	\end{equation}
	In particular, $h_1$ and $h_2$ satisfies \eqref{eq:ellcond}	
	and so we can apply Theorem \ref{th:flatpartbound} to get (see \eqref{eq:flatpartbound2})
	\begin{equation*}\label{eq:lhjk}
	\max \left\{   \oeps   ,h_1 K  , \ell h_2 \right\}  \geq 
	\delta K   \min\left\{  \frac{1}{\exp \left( \frac{C^4 \lambda_1^4 \lambda_2^4 K^8}{ \ell^8} \right) -1} , 1\right\}
	\end{equation*}
	Furthermore, under conditions \eqref{eq:h1h2} we can use 
	Lemma \ref{lem:opsibound} to write
	\[
        K \geq \left( \frac{\delta \ell}{\lambda_1\lambda_2}\right)^{1/2}.
	\]
	It follows that (recall that $D=\mathrm{diam}\, \Omega_1$),
	\[ 
	\frac{C^4 \lambda_1^4 \lambda_2^4 K^8}{\ell^8} \geq C^4 \left( \frac \delta \ell\right)^4 \geq C^4 \left( \frac \delta D \right)^4 
	\]
	and so
	\[
	\frac{1}{\exp \left( \frac{C^4 \lambda_1^4 \lambda_2^4 K^8}{ \ell^8} \right) -1} \leq
	\frac{1}{\exp \left(  C^4 \left( \frac \delta D \right)^4 \right) -1}.
	\]
	We deduce (using \eqref{eq:lhjk}) that there exists a constant $C_0$, depending on $D$, $\delta$ and the dimension such that
	\[
	\max \left\{   \oeps   , K h_1 , D  h_2 \right\}  \geq 
	\frac  {1}  {C_0}  \frac{K}{\exp \left( \frac{C^4 \lambda_1^4 \lambda_2^4 K^8}{ \ell^8}  \right)-1}.
	\]
        We now observe the following elementary fact:
        \begin{equation}
\forall a>0,\qquad u\mapsto \frac{u}{\exp \left( a u^8 \right) -1}\ \mbox{is monotone decreasing in } u.\label{stat:ospilin}
        \end{equation}
        This implies in particular that for all $\ell$ we have
\[
 \frac{\opsi }{\exp \left( \frac{C^4 \lambda_1^4 \lambda_2^4 \opsi^8}{ \ell^8} \right) -1}
 \geq  \frac{L_\infty }{\exp \left( \frac{C^4 \lambda_1^4 \lambda_2^4 L_\infty^8}{ \ell^8} \right) -1},
 \]
and so
 \begin{equation}\label{eq:bbnm}
	\max \left\{   \oeps   , K h_1 , D  h_2 \right\}  \geq 
	\sigma(\ell) : = \frac  {1}  {C_0}  \frac{L_{\infty} }{\exp \left( \frac{C^4 \lambda_1^4 \lambda_2^4 L_{\infty}^8}{ \ell^8}  \right)-1},
	\end{equation}
	where the function $\sigma(\ell) $ is monotone increasing and satisfies $\lim_{\ell\to 0^+} \sigma(\ell)=0$.
In other words
	\begin{equation*}\label{eq:eoo}
	\mbox{either }  \frac{\oeps}{\ell} \geq \frac{\sigma(\ell)}{\ell} \mbox{ or } K h_1 \geq \sigma(\ell) 
	\mbox{ or } D h_2 \geq \sigma(\ell). 
	\end{equation*}
	Since both functions $\ell \mapsto \sigma(\ell)$ and $\ell \mapsto \frac{\sigma(\ell)}{\ell}$
	are monotone increasing (for the second one, this is a consequence of \eqref{stat:ospilin} again),
	we can introduce their inverses $\sigma_1$ and $\sigma_2$.
	The conditions above 
	are then equivalent to
	\[
        \ell \leq \max\left\{ \sigma_2\left(\frac{\oeps}{\ell}\right), \sigma_1(Kh_1),\sigma_1(D h_2)\right\}.
	\]
Combining this with \eqref{eq:cond3}, we deduce that for all $\ell>0$ we have
	\[ 
	\ell 
	\leq \max \left\{   \sigma_2\left(\frac{\oeps}{\ell} \right) , \max\left\{ \sigma_1(K h_1), 2h_1\right\}  , \max \left\{ \sigma_1( D h_2), \sqrt{C  \lambda_1 \lambda_2 \, L_\infty h_2},\frac{\lambda_1\lambda_2}{\delta} h_2 \right\} \right\} 
	\]
	and the general result follows, recalling that $\ell=|x-x'|$ and that Lemma \ref{lemma:sub} gives
	$|z-z'| \geq 2 \frac{\oeps }{\ell}$.

        It remains to treat the special case $h_1=h_2=0$,  where we immediately obtain that
\[
|x-x'|\leq \sigma_2(|z-z'|),
\]
proving that for any given $z$ the sub-differential of $\psi^*$ is always reduced to one point (take $z=z'$ and any $x,\;x'\in \partial\psi^*(z)$). Consequently $\psi^*$ is $C^1$ as claimed.

To bound $\sigma_2$, we trivially observe that 
\[
\frac{u}{\exp(a\,u^8)-1}\leq 2\,\frac{a^{-1/8}}{\exp(a\,u^8/2)-1}.
\]
Consequently for some numerical constant $\tilde C$
\[
\sigma_2^{-1}(\ell)=2\,\frac{\sigma(\ell)}{\ell}\geq \frac  {1}  {\tilde C}  \frac{\lambda_1^{-1/2}\,\lambda_2^{-1/2}}{\exp \left( \frac{\tilde C^4 \lambda_1^4 \lambda_2^4 L_{\infty}^8}{ \ell^8}  \right)-1}.
\]
Therefore for some $\tilde C$
\[
\sigma_2(w)\leq \tilde C\,\frac{\sqrt{\lambda_1\,\lambda_2}\,L_\infty}{\left(\log\left(1+\frac{1}{\tilde C\,\sqrt{\lambda_1\,\lambda_2}\,w}\right)\right)^{1/8}},
\]
which concludes the proof.
\end{proof}

\begin{proof}[Proof of lemma \ref{lemma:sub}]
	By definition of $\oeps $, there exists $y \in (x,x')$, with $y=tx+(1-t)x'$ for some $t \in (0,1)$ such that 
	\begin{equation}
	\label{eq:affinemin}
	\psi(y) = t \psi(x)+(1-t) \psi(x') - \oeps. 
	\end{equation}
	By the definition of subdifferential we have that
	\begin{align}
	\label{eq:subdiffmin}
	\psi(z) & \geq \psi(x') + y' \cdot (z-x') \\
	\psi(z) & \geq \psi(x') + y'' \cdot (z-x'') \mbox{.}\nonumber
	\end{align}
	Plugging (\ref{eq:affinemin}) into the inequalities (\ref{eq:subdiffmin}) yields
	\begin{align*}
	y' \cdot (x'-x'') & \geq \psi(x')- \psi(x'') +\frac{\oeps }{1-t} \\
	-y'' \cdot (x'-x'') & \geq \psi(x'')-\psi(x') + \frac{\oeps }{t} \mbox{,}
	\end{align*} 
	so that finally, by adding both inequalities, we get
	\begin{equation*}
	2 \ell |y'-y''|\geq (y'-y'') \cdot (x'-x'') \geq \frac{\oeps }{1-t}+ \frac{\oeps }{t} \geq 4 \oeps  \mbox{,}
	\end{equation*}
	which concludes the proof.
\end{proof}


\section{Proof of Theorem \ref{thm:optimal}}\label{sec:proof}
Theorem \ref{thm:optimal} follows from the following result together with well known facts from the theory of optimal transportation:
\begin{theorem}\label{thm:sub}
Let  $\mu,\;\nu$ be two probability measures on $\R^n$. 
Assume $\Gamma\subset \R^{2n}$ is a closed set such that for any Borel set $O$ there holds:
  \begin{equation}\begin{split}
  &\mu(O)\leq \nu\left(\{y\in \R^n\,|\;\exists x\in O,\ (x,y)\in \Gamma\}\right),\\
  &\nu(O)\leq \mu\left(\{x\in \R^n\,|\;\exists y\in O,\ (x,y)\in \Gamma\}\right).
\end{split}\label{assumemunu}
  \end{equation}
Then there exists $\pi\in\Pi(\mu,\nu)$ concentrated on $\Gamma$ (that is $\mathrm{Supp}(\pi) \subset \Gamma$).
\end{theorem}

\begin{proof}[Proof of Theorem \ref{thm:optimal}]
We note that when $\Gamma = \mathrm{Graph}\,{\pa \psi}$ for some convex function $\psi$, then \eqref{assumemunu} is equivalent to the conditions \eqref{eq:ineq1} and \eqref{eq:ineq2}.
Theorem \ref{thm:sub} thus implies that there exists $\pi\in\Pi(\mu,\nu)$ such that $\mathrm{Supp}(\pi) \subset \Gamma$.
The result then follows from a classical result of measure theory  (see for example Theorem 2.12 in \cite{villani2003topics}).
\end{proof}

We now turn to the proof of Theorem \ref{thm:sub}. The first step is to prove the result when $\mu$ and $\nu$ are both sums of Dirac masses with identical mass:
$$ \mu = \frac 1 N \sum_{i=1}^N \delta_{x_i}, \quad  \nu = \frac 1 N \sum_{j=1}^N \delta_{x_j}.$$
In that case, we define the $N\times N$ matrix $A=(A_{ij})$ by
$$ A_{ij} = 
\left\{
\begin{array}{ll}
1 & \mbox{ if } (x_i,y_j)\in \Gamma, \\
0 & \mbox{ otherwise}
\end{array}
\right.
$$
and  Theorem \ref{thm:sub} is equivalent to 
\begin{proposition}\label{prop:sub1}
Let $A$ be a $N\times N $ matrix  with $A_{ij}\in \{0,\;1\}$ and such that for any $I,\, J$ subsets of  $\{1,\ldots,N\}$, we have
  \begin{equation}
|I|\leq |\{j\,|\;\sum_{i\in I} A_{ij}>0\}|,\quad |J|\leq |\{i\,|\;\sum_{j\in J} A_{ij}>0\}|. \label{enoughmass}
\end{equation}
Then there exists a stochastic matrix $\pi=(\pi_{ij})$ such that 
$\pi_{ij} \in \{0,1\}$, $\sum_j \pi_{ij}=1$, $\sum_i \pi_{ij}=1$ and satisfying $\pi_{ij}=0$ whenever $A_{ij}=0$. 
  \end{proposition}
\begin{proof}[Proof of Proposition \ref{prop:sub1}]
The proof proceeds by induction on $N$, the result being obvious for $N=1$.
We distinguish two cases: 
Whether there are strict subsets $I_0$ or $J_0$ for which there is equality in \eqref{enoughmass} or not.

\bigskip

\noindent {\bf Case 1:} 
We assume that for all strict subsets $I$ or $J$ of $\{1,\dots,N\}$, we have a strict inequality in \eqref{enoughmass}, that is
  \[
  |I|< |\{j\,|\;\sum_{i\in I} A_{ij}>0\}|,\quad |J|< |\{i\,|\;\sum_{j\in J} A_{ij}>0\}|.
  \]
We then choose any $i_0,\,j_0$ such that $A_{i_0j_0}=1$. Up to relabeling the rows and columns of $A$, we can always assume that we can take $i_0=j_0=N$ and we  consider the $N-1\times N-1$ matrix $\tilde A$ consisting of the first $N-1$ rows and columns of $A$.
We claim that $\tilde A$ satisfies \eqref{enoughmass}:
Indeed, for any $I\subset\{1,\ldots,N-1\}$, by applying \eqref{enoughmass} to $A$, we have that
  \begin{align*}
|I|
& \leq |\{j\in \{1,\ldots,N\}\,|\;\sum_{i\in I} A_{ij}>0\}|-1\\
& \leq |\{j\in \{1,\ldots,N-1\}\,|\;\sum_{i\in I} A_{ij}>0\}|\\
& =|\{j\,|\;\sum_{i\in I} \tilde A_{ij}>0\}|.
\end{align*}
and a similar inequality for $J\subset\{1,\ldots,N-1\}$.

Therefore by induction, there exists a stochastic $(N-1)\times (N-1)$ matrix $\tilde \pi$ with $\tilde \pi \in \{0,1\}$, $\tilde\pi_{ij}=0$ if $\tilde A_{ij}=0$ and $\sum_i \tilde\pi_{ij}=1$, $\sum_j \tilde\pi_{ij}=1$. 
We can then define the matrix $\pi$ by setting $\pi_{ij}=\tilde \pi_{ij}$ if $i\leq N-1$ and $j\leq N-1$, 
$\pi_{N N}=1$ and $\pi_{ij}=0$ otherwise. It is straightforward to check that $\pi$ satisfies all requirements.

\bigskip

\noindent {\bf Case 2:} 
We assume that  there exists a strict subset $I_0$ or $J_0$ of $\{1,\dots,N\}$ for which we have equality in  \eqref{enoughmass}. For example, we assume  that there is a strict subset $I_0$ such that
\begin{equation}\label{eq:I0}
|I_0|= |\{j\,|\;\sum_{i\in I_0} A_{ij}>0\}|
\end{equation}
and we denote 
$J_0=\{j\,|\;\sum_{i\in I_0} A_{ij}>0\}$. 


Since $|I_0|=|J_0|$, we can define the square matrices $P$ and $Q$ by
\[
P_{ij}=A_{ij}\,\mathbb{I}_{i\in I_0, j\in J_0},\quad Q_{ij}=A_{ij}\,\mathbb{I}_{i\in I_0^c, j\in J_0^c}. 
\]
and we are going to show that $P$ and $Q$ satisfy \eqref{enoughmass} on $I_0\times J_0$, and $I_0^c\times J_0^c$ respectively.

We start with $P$: for any $I\subset I_0$, \eqref{enoughmass} implies:
\[
|I|\leq |\{j\,|\;\sum_{i\in I} A_{ij}>0\}|.
\]
The definition of $J_0$ implies that $A_{ij}=0$ if $i\in I_0$ and $j\not\in J_0$. Since $I\subset I_0$, we can thus write:
\[
\{j\,|\;\sum_{i\in I} A_{ij}>0\}=\{j\in J_0\,|\;\sum_{i\in I} A_{ij}>0\}=\{j\in J_0\,|\;\sum_{i\in I} P_{ij}>0\},
\]
and we deduce that
\begin{equation}
|I|\leq |\{j\in J_0\,|\;\sum_{i\in I} P_{ij}>0\}| \qquad \mbox{ for all $I\subset I_0$}\label{enoughmassM1}
\end{equation}
which shows that $P$ satisfies the first inequality in \eqref{enoughmass}. To prove the second inequality, we proceed by contradiction: Assume that there is a subset $J\subset J_0$ such that
\[
|J|> |\{i\in I_0\,|\;\sum_{j\in J} P_{ij}>0\}|  = |\{i\in I_0\,|\;\sum_{j\in J} A_{ij}>0\}|,
\]
then denote $\tilde I=\{i\in I_0\,|\;\sum_{j\in J} A_{ij}>0\}$ and define $I'=I_0\setminus \tilde I$, $J'=J_0\setminus J$. Since $|J|>|\tilde I|$, we have that
\begin{equation}\label{eq:I'}
|I'|=|I_0|-|\tilde I|>|I_0|-|J|=|J_0|-|J|=|J'|.
\end{equation}
Consider any $j'$ s.t. $\sum_{i\in I'} A_{ij'}>0$. Since $A_{ij'}=0$ for $i\in I_0$ and $j'\not \in J_0$, we have that $j'\in J_0$. Next, let $i'\in I'$ be such that $A_{i'j'}>0$ (which exists by the choice of $j'$). Since $i'\in I'= I_0\setminus \tilde I$, the definition of $\tilde I$ implies that $A_{i'j}=0$ for all $j\in J$. So we must have   $j'\not\in J$. We thus have $j'\in J'$.
We just proved that
\[
\{j'\,|\;\sum_{i'\in I'} A_{i'j'}>0\}=\{j'\in J_0\,|\;\sum_{i'\in I'} A_{i'j'}>0\}\subset J'.
\]
Together with \eqref{eq:I'}, this implies that
\[
|I'|>|J'|\geq |\{j'\,|\;\sum_{i'\in I'} A_{i'j'}>0\}|,
\]
which contradict \eqref{enoughmass}. We can then conclude that $P$ satisfies \eqref{enoughmass} on $I_0\times J_0$.
\medskip

We now turn to $Q$ and start by considering any $J\subset J_0^c$. We recall that $A_{ij}=0$ for $i\in I_0$ and $j\in J\subset J_0^c$ and since $\{1,\ldots,N\}=I_0\cup I_0^c$ we have:
\[
\{i\,|\;\sum_{j\in J} A_{ij}>0\}=\{i\in I_0^c\,|\;\sum_{j\in J} A_{ij}>0\}.
\]
Applying \eqref{enoughmass} on $A$ for $J$, we get
\[
|J|\leq |\{i\,|\;\sum_{j\in J} A_{ij}>0\}|=|\{i\in I_0^c\,|\;\sum_{j\in J} Q_{ij}>0\}|,
\]
proving the corresponding half of \eqref{enoughmass} for $Q$.

Next, for any $I\subset I_0^c$, we apply \eqref{enoughmass} with $\bar I=I_0\cup I$:
\[
|I_0|+|I|=|\bar I|\leq |\bar J|,\quad \bar J=\{j\,|\;\sum_{i\in \bar I} A_{ij}>0\}.
\]
We recall that $J_0=\{j\,|\;\sum_{i\in I_0} A_{ij}>0\}$ and denote 
$J=\{j\in J_0^c\,|\;\sum_{i\in I} A_{ij}>0\}$.
We have $J\cup J_0 \supset \{j \,|\;\sum_{i\in I} A_{ij}>0\}$, and 
since
\[
\bar J=\{j\,|\;\sum_{i\in I_0} A_{ij}>0\}\bigcup \{j\,|\;\sum_{i\in I} A_{ij}>0\}.
\]
we have $\bar J=J_0\cup J$.
This implies that $|I_0|+|I|\leq |\bar J|\leq |J_0|+|J|$, that is
$$
 |I|\leq |J|=|\{j\in J_0^c\,|\;\sum_{i\in I} A_{ij}>0\}|,
$$
which completes the proof that $Q$ satisfies \eqref{enoughmass}.
\medskip

We can now complete the proof:
Since $I_0\neq\emptyset$ and $I_0\neq \{1,\ldots,N\}$, $P$ and $Q$ have dimensions strictly less than $N$ and we may apply our induction assumption. This gives us $p_{ij}$ on $I_0\times J_0$ and $q_{ij}$ on $I_0^c\times J_0^c$, stochastic matrices,
$$
\ \sum_j p_{ij}=1\quad \forall i\in I_0,\qquad  \sum_{j} q_{ij}=1\quad \forall i\in I_0^c,$$
$$
\ \sum_j p_{ij}=1\quad  \forall j\in J_0,\qquad  \ \sum_j q_{ij}=1 \quad \forall j\in J_{0}^c,
$$
with $p_{ij}=0$ if $P_{ij}=0$ and $q_{ij}=0$ if $Q_{ij}=0$. We simply extend $p$ and $q$ by 0 on the whole $\{1,\ldots,N\}^2$ and define $\pi=p+q$. 

\end{proof}

Proposition \ref{prop:sub1} proves Theorem \ref{thm:sub} when the measures $\mu$ and $\nu$ are 
both sums of Dirac masses with identical mass.
Our next step is to extend that result to general sums of Dirac masses.

\begin{proposition}\label{prop:sub2}
Assume that 
  \[
\mu=\sum_{i=1}^{M_1} m_i\,\delta_{x_i},\quad \nu=\sum_{j=1}^{M_2} n_j\,\delta_{y_j}
  \]
for some points $x_1,\ldots,x_{M_1}$ and $y_1,\ldots, y_{M_2}$ in $\R^n$ and some (positive) masses $m_1,\ldots,m_{M_1}$, and $n_1,\ldots,n_{M_2}$.
Then Theorem \ref{thm:sub}  holds.
\end{proposition}

\begin{proof}[Proof of  Proposition \ref{prop:sub2}]
By splitting points if needed, we can always assume that $M_1=M_2=M$ and that $m_i, n_j>0$ for all $i$, $j$.
We can also assume that $\Gamma$ is concentrated on $\bigcup_{i,j} \{(x_i,y_j)\}$. For this reason, we define $\gamma$ as the set of indices $(i,j)$ s.t. $(x_i,x_j)\in \Gamma$. In that discrete setting, Assumption \eqref{assumemunu} is then equivalent  to   
\begin{equation}
  \begin{split}
&\forall I\subset \{1,\ldots,M\},\quad \sum_{i\in I} m_i\leq \sum_{\{j\,|\,\exists i\in I\ s.t.\ (i,j)\in \gamma\} } n_j,\\  
& \forall J\subset \{1,\ldots,M\},\quad \sum_{i\in J} n_i\leq \sum_{\{i\,|\,\exists j\in J\ s.t.\ (i,j)\in \gamma\}} m_j.
\end{split}\label{assumemunudiscrete}
  \end{equation}
In order to use the result of Proposition \ref{prop:sub1}, we want to approximate $\mu$ and $\nu$ by measures 
 $\mu_N$, $\nu_N$ that are sums of Dirac masses with identical mass.
 To do this, given $N\geq 2\,M$, we replace the mass $m_i$ at $x_i$ (respectively the mass $n_j$ at $y_j$) into $k(i)$ masses (respectively $l(j)$ masses) $\frac 1 N$ all located at that same point, with $k(i)$ and $l(j)$ such that
$$ m_i-\frac{1}{N}\leq \frac{k(i)}{N}\leq m_i,\quad n_j-\frac{1}{N}\leq \frac{l(j)}{N}\leq n_j$$
(we can always assume that $1/N\leq \inf(\inf_i m_i,\;\inf_j n_j)$ so that $k(i),l(j)>0$). 
We note that we have some left over mass $m_i-\frac{k(i)}{N}$ at each points. By summing over all $i$ (and all $j$), the total left over masses are 
\[
 1- \sum_i \frac{k(i)}{N} = \frac{k(0)}{N}  \in\left(0, \frac{M}{N}\right) ,\quad 1-\sum_j \frac{l(j)}{N} = \frac{l(0)}{N} \leq 1\in\left(0, \frac{M}{N}\right).
\]
where $k(0)=  N- \sum_i k(i)\leq M$ and $l(0)=  N- \sum_j l(j)\leq M$.
We thus add $k(0)$ (resp. $l(0)$) masses $1/N$ at some point $x_0\neq x_i$ (resp. $y_0\neq y_j$). 

We can write
\[
\mu_N=\frac{1}{N}\sum_{k=1}^{N} \delta_{\bar x_k},\quad \nu_N=\frac{1}{N}\sum_{l=1}^{N} \delta_{\bar y_l}
\]
where the points $\bar x_k$ (resp. $\bar y_l$) are the same points as the $x_i$ or the additional distinct point $x_0$ (resp. $y_i$ and $y_0$) which we just subdivide. It is useful to introduce $K(i)$ (resp. $L(j)$), the set of indices $k$ such that $\bar x_k=x_i$ (resp. $\bar y_k=y_j$).
We have in particular $|K(i)| = k(i)$ and $|L(j)|=l(j)$.


Observe that $\mu_N$ and $\nu_N$ converge strongly to $\mu$ and $\nu$ when $N\to\infty$ since
\[
\int d|\mu_N-\mu|\leq \frac{2M}{N},\quad \int d|\nu_N-\nu|\leq \frac{2M}{N}
\]
(since there is a mass discrepancy of at most $1/N$ at each of the $M$ points $x_1,\dots,x_M$ and an additional mass of at most $M/N$ at $x_0$).

We now define $\gamma_N$ as
\[
\begin{split}
  \gamma_N=&\left(\bigcup_{(i,j)\in \gamma} K(i)\times L(j)\right)\bigcup \left(K(0)\times (\{1,\ldots,N\}\setminus L(0))\right)\\
  &\bigcup\left((\{1,\ldots,N\}\setminus K(0))\times L(0)\right). 
\end{split}\]
The first component of $\gamma_N$ is the natural extension from $\gamma$: If $x_i$ and $y_j$ were connected, then any $\bar x_k$ s.t. $\bar x_k=x_i$ is still connected to any $\bar y_l$ with $\bar y_l=y_j$. Because we lose a bit of mass on the points $\bar x_k$ and $\bar y_l$, this may not be enough to ensure that \eqref{assumemunudiscrete} holds on $\mu_N$ and $\nu_N$.
For this reason, $K(0)$ and $L(0)$ serve as a mass reservoir: $K(0)$ is connected to all $l\geq 1$ and $L(0)$ to all $k\geq 1$, but of course $K(0)$ is not connected with $L(0)$.

We now check that \eqref{assumemunudiscrete} holds for this $\gamma_N$. Let $I$ be a subset of $\{1,\dots , N\}$. We consider three cases:

\noindent\underline{If $I\subset K(0)$} then $\{l\,|\, \exists k\in I\ s.t.\ (k,l)\in \gamma_N\}=\{1,\ldots,N\}\setminus L(0)$. Therefore
\[
\mu_N(I)\leq \mu_N(K(0))\leq \frac{M}{N}\leq 1-\frac{M}{N}=\nu_N(\{l,\;\exists k\in I\ s.t.\ (k,l)\in \gamma_N\}),
\]
since $N\geq 2\,M$.

\noindent\underline{If $I\cap K(0)\neq \emptyset$ but $I\not\subset K(0)$}, then $\{l\,|\, \exists k \in I\ s.t.\ (k,l)\in \gamma_N\}=\{1,\ldots,N\}$. Trivially
\[
\mu_N(I)\leq 1=\nu_N(\{l,\;\exists k\in I\ s.t.\ (k,l)\in \gamma_N\}).
\]

\noindent\underline{If $I\cap K(0)=\emptyset$}, then denote $\bar I=\{i\,|\,\ K(i)\cap I\neq\emptyset\}$. Observe that if $(k,l)\in \gamma_N$ and $k\in K(i),\;l\in L(j)$ then for any $k'\in K(i),\; l'\in L(j)$ one also has that $(k',l')\in \gamma_N$ by the definition of $\gamma_N$.  Therefore
\begin{equation}
\{l\,|\,\exists k\in I\ s.t.\ (k,l)\in \gamma_N\}=L(0) \bigcup\left(\bigcup_{j,\; \exists i\in \bar I\ s.t.\ (i,j)\in \gamma} L(j)\right).\label{explicitinggammaNonI}
\end{equation}
Consequently by the definition of $\bar I$ we have
\[
\mu_N(I)=\frac{|I|}{N}\leq \sum_{i\in \bar I} \frac{k(i)}{N}
\]
and since $k(i)\leq N \, m_i$ from the construction of $\mu_N$ we deduce
\[
\mu_N(I)\leq \sum_{i\in \bar I} \frac{k(i)}{N}\leq \sum_{i\in \bar I} m_i=\mu(\bar I).
\]
Applying \eqref{assumemunu} to $\mu$, we get
\[
\mu_N(I)\leq \mu(\bar I)\leq \nu(\{j,\; \exists i\in \bar I\ s.t.\ (i,j)\in \gamma\})=\sum_{j,\; \exists i\in \bar I\ s.t.\ (i,j)\in \gamma} n_j.
\]
From  the construction of $\nu_N$, we have $n_j\geq \frac{l(j)}{N}$ and
\[
\begin{split}
  \mu_N(I)&\leq \sum_{j,\; \exists i\in \bar I\ s.t.\ (i,j)\in \gamma} \frac{l(j)}{N}+\sum_{j,\; \exists i\in \bar I\ s.t.\ (i,j)\in \gamma} \left(n_j-\frac{l(j)}{N}\right)\\
  &\leq \sum_{j,\; \exists i\in \bar I\ s.t.\ (i,j)\in \gamma} \frac{l(j)}{N}+\sum_{j=1}^M \left(n_j-\frac{l(j)}{N}\right)\\
  &\leq \sum_{j,\; \exists i\in \bar I\ s.t.\ (i,j)\in \gamma} \frac{l(j)}{N}+1-\sum_{j=1}^M \frac{l(j)}{N}\\
  &\leq \sum_{j,\; \exists i\in \bar I\ s.t.\ (i,j)\in \gamma} \frac{l(j)}{N}+\frac{l(0)}{N}
\end{split}
\]
Using \eqref{explicitinggammaNonI}, we deduce
 \[
\mu_N(I)\leq \nu_N(\{l,\;\exists k\in I\ s.t.\ (k,l)\in \gamma_N\}),
\]
which proves that \eqref{assumemunu} holds for the measures $\mu_N$, $\nu_N$ and the set $\gamma_N$.
\medskip

We now apply Proposition \ref{prop:sub1} to $\mu_N$ and $\nu_N$. We obtain $\pi_N$ a transference plan
\[
\pi_N=\sum_{k,l} \frac{\pi_{k,l}^N}{N} \delta_{(\bar x_k,\;\bar y_l)},\quad \sum_l \pi_{k,l}^N=1,\quad \sum_k \pi_{k,l}^N=1,\quad \pi^N_{k,l}=0\ \mbox{if}\ (k,l)\not\in \gamma_N.
\]
Because the $\bar x_k$ and $\bar y_k$ are equal to the $x_i,\; y_j$ or to $x_0$, $y_0$, we can also express $\pi_N$ as
\begin{equation}
\pi_N=\sum_{i,j\geq 1} \bar \pi_{i,j}^N\, \delta_{(x_i,\;y_j)}+\sum_{j\geq 1} \bar\pi^N_{0,j}\,\delta_{( x_0,\; y_j)}+\sum_{i\geq 1} \bar\pi^N_{0,j}\,\delta_{( x_i,\; y_0)}.\label{piNfixedpoints}
\end{equation}
Moreover for any $i\geq 1$, or any $j\geq 1$
\[
\sum_{j\geq 1} \bar \pi_{ij}^N+\pi_{i0}^N=\frac{k(i) }{N}, \quad \sum_{j\geq 1} \bar \pi_{ij}^N+\pi_{i0}^N=\frac{l(i)}{N}.
\]
By the construction of $\mu_N,\;\nu_N$, this yields that
\begin{equation}
m_i-\frac{1}{N}\leq\sum_{j\geq 1} \bar \pi_{ij}^N+\pi_{i0}^N\leq m_i,\quad n_i-\frac{1}{N}\leq\sum_{i\geq 1} \bar \pi_{ij}^N+\pi_{0j}^N\leq n_i.\label{boundpiN1}
\end{equation}
On the other hand,
\begin{equation}
  \sum_j \bar \pi_{0j}^N=\frac{k(0)}{N}\leq \frac{M}{N},\quad \sum_i \bar \pi_{i0}^N=\frac{l(0)}{N}\leq \frac{M}{N}.
  \label{boundpiN0}
\end{equation}
Since the points $x_i,\;y_j$ together with $x_0,\;y_0$ are fixed, we may simply pass to the limit in $\pi_N\to \pi$ using \eqref{piNfixedpoints} by extracting subsequences such that all $\bar\pi_{ij}^N\to \bar \pi_{ij}$ for all $i,\;j\geq 0$. By \eqref{boundpiN0}, we have that $\bar \pi_{0j}=\bar \pi_{i0}=0$ so that
\[
\pi=\sum_{i,j\geq 1} \bar \pi_{i,j}\, \delta_{(x_i,\;y_j)}.
\]
By \eqref{boundpiN1} and given that $\bar\pi_{i0}=\bar\pi_{0j}=0$, we get
\[
\sum_{j\geq 1} \bar \pi_{ij}=m_i,\quad \sum_{i\geq 1} \bar \pi_{ij}=n_j,
\]
and so $\pi$ is a transference plan between $\mu$ and $\nu$. It only remains to check that $\pi$ is concentrated on $\Gamma$: Given any $(i,j)\not\in \gamma$, then for any $N$ and any $k\in K(i),\;l\in L(j)$, we have that $(k,l)\not\in \gamma_N$. Therefore $\pi^N_{k,l}=0$ and 
\[
\bar\pi_{ij}^N=\sum_{k\in K(i),\;l\in L(j)} \frac{\pi^N_{kl}}{N}=0,
\]
so that we also have that $\bar \pi_{ij}=0$, thus  proving that $\pi$ is indeed concentrated on $\Gamma$.

\end{proof}

We are now ready to prove Theorem \ref{thm:sub}.

\begin{proof}[Proof of Theorem \ref{thm:sub}]
  First of all by density, we can assume that both $\mu$ and $\nu$ are compactly supported on some large ball $B(0,R)$.

Define $\Gamma_x$ and $\Gamma_y$ the projections of $\Gamma$,
  \[
\Gamma_x=\{x\in \R^n\,|\;\exists y\in \R^n,\ (x,y)\in \Gamma\},\quad \Gamma_y=\{y\in \R^n\,|\;\exists x\in \R^n,\ (x,y)\in \Gamma\}.
\]
Of course $\Gamma_x,\;\Gamma_y$ are closed. Moreover $\mu$ is supported on $\Gamma_x$: For any open set $O$ with $O\cap\Gamma_x=\emptyset$ then for any $y\in \R^n$ and any $x\in O$, $(x,\,y)\not\in \Gamma$ so
\[
\{y\in \R^n\,|\;\exists x\in O,\ (x,y)\in \Gamma\}=\emptyset,
\]
and therefore $\mu(O)=0$ by assumption \eqref{assumemunu}. Similarly $\nu$ is supported on $\Gamma_y$.
\medskip

For any $k\in \mathbb{N}$, we define the hypercubes $C_i^k$ of diameter $2^{-k}$, centered at points $x_i \in 2^{-k} \mathbb{Z}^n$  and that cover a fixed selected hypercube $C_0$ s.t. $B(0,R)\subset C_0$. This decomposition is obviously hierarchical since $C_i^k$ is composed of exactly $2^{n}$ small hypercubes $C_j^{k+1}$.

By shifting the hypercubes if necessary, we may assume that $\mu\left(\bigcup_i \partial{C}_{i}^k \right)=0$. For any $k$, we define an approximation $\mu_N$ with $N=C\,2^{kn}$ points,
\[
\mu_N=\sum_{i=1}^N m_i^N\,\delta_{x_i},\quad m_i^N=\mu(C_i^k).
\]
We also define $\nu_N$ in the same manner. Both $\mu_N$ and $\nu_N$ remain probability measures since $\sum_i m_i^N=\sum_i \mu(C_i^k)=\mu(\bigcup C_i^k)$ as $\mu\left(\bigcup_i \partial {C}_{i}^k \right)=0$.

Finally we define $\Gamma_N$ as the union $\bigcup_{(i,j)\in \gamma_N} C_i^k\times C_j^k$ where
\[
\gamma_N=\{(i,j)\,|\; \exists x\in C_i^k,\ \exists y\in C_j^k,\ (x,y)\in \Gamma\}.
\]
We observe of course that $\Gamma\subset \Gamma_N$ and that $d(\Gamma_N,\Gamma)\leq C\,2^{-k}\to 0$ as $N\to \infty$.

Consider now any subset $I$ of indices $i$ and define $O=\bigcup_{i\in I} C_i^k$.
By our construction and assumption \eqref{assumemunu}
\[
\mu(O)=\sum_{i\in I} m_i^N\leq \nu\left(\{y\in \R^n\,|\;\exists x\in O,\ (x,y)\in \Gamma\}\right).
\]
By the definition of $O$, we have
\[
\{y\in \R^n\,|\;\exists x\in O,\ (x,y)\in \Gamma\}=\bigcup_{i\in I} \{y\in \R^n\,|\;\exists x\in C_i^k,\ (x,y)\in \Gamma\}.
\]
And since $\Gamma\subset \Gamma_N$, we deduce
\[\begin{split}
\{y\in \R^n\,|\;\exists x\in O,\ (x,y)\in \Gamma\}&\subset \bigcup_{i\in I} \{y\in \R^n\,|\;\exists x\in C_i^k,\ (x,y)\in \Gamma_N\}\\
&\subset\bigcup_{\{j\,|\, \exists i\in I\;s.t.\;(i,j)\in \gamma_N\}} C_j^k,
\end{split}
\]
by the definition of $\gamma_N$. Hence, since $\nu\left(\bigcup_i \partial {C}_{i}^k \right)=0$
\[
\begin{split}
  \nu\left(\{y\in \R^n\,|\;\exists x\in O,\ (x,y)\in \Gamma\}\right)& \leq \sum_{\{j\,|\, \exists i\in I\;s.t.\;(i,j)\in \gamma_N\}} \nu(C_j^k)= \sum_{\{j\,|\, \exists i\in I\;s.t.\;(i,j)\in \gamma_N\}} n_j^N 
\end{split}
\]
Therefore
\[
\sum_{i\in I} m_i^N\leq 
\sum_{\{j\,|\, \exists i\;s.t.\;(i,j)\in \gamma_N\}} n_j^N
\]
which is the first inequality in \eqref{assumemunudiscrete}.
The proof is similar when reversing the roles of $\mu_N$ and $\nu_N$ and this allows us to conclude that $\mu_N$ and $\nu_N$ satisfy \eqref{assumemunu} with $\Gamma_N$.

\medskip

We can thus apply Proposition \ref{prop:sub2} to get the existence of a transference plan $\pi_N$ concentrated on $\Gamma_N$ and with marginals $\mu_N$ and $\nu_N$.

\medskip

By the tightness of $\pi_N$, we can extract a converging subsequence (still denoted by $N$ for simplicity) s.t. $\pi_N \to \pi$ for the weak-* topology of measures.
Trivially $\pi$ has marginals $\mu$ and $\nu$. 

\medskip

To conclude the proof, we need to show that $\pi$ is concentrated on $\Gamma$: Consider any open set $O$ with $O\cap \Gamma=\emptyset$ and any continuous function $\phi$ on $\R^{2n}$ with compact support on $O$.

We claim that $\sup_{\Gamma_N} |\phi|\to 0$ as $N$ tends to $\infty$. Indeed assume, by contradiction, that there exists $(x_N,\;y_N)\in \Gamma_N$ and $\eta>0$ s.t. $\phi(x_N,y_N)>\eta$. Since $\Gamma_N\subset B(0,2R)$ then we can extract converging subsequences $x_N\to x$ and $y_N\to y$. Since $d(\Gamma_N,\Gamma)\to 0$ then $(x,y)\in \Gamma$ but, since $\phi$ is continuous, we have that $\phi(x,y)>\eta$. Recalling that $\phi$ has compact support in $O$ with $O\cap \Gamma=\emptyset$ gives a contradiction.

We thus have
\[
\int \phi\,d\pi_N=\int_{\Gamma_N} \phi\,d\pi_N\leq \sup_{\Gamma_N} |\phi|\longrightarrow 0,\quad \mbox{as}\ N\to \infty,
\]
which gives $\int \phi\,d\pi=0$, proving that $\pi$ is concentrated on $\Gamma$ and completing the proof of Theorem~\ref{thm:sub}.
  \end{proof}

\appendix

\section{Proof of Theorem \ref{th:strictConvexity2D}}\label{app:2d}
As in the proof of our main theorem, we denote by $x=(\xp,\xo)$ the points in $\Omega \subset \RR^2$
where $\xp$ is the coordinate along the line $H$ and $\xo$ the orthogonal coordinate.
We then have (the proof is similar to that of Lemma 	\ref{lemma:SubdifferentialBound}) that for all $x\in \Omega_1$ such that $|x_\pl|\leq \ell/2$ there holds
\begin{equation}\label{eq:psixp}
|\pa_{\xp}\psi(x) | \leq \frac{2K }{\ell}\, \left(|x_\pe|+\gamma \right).
\end{equation}

Next, we note that the fact that $\det D^2 \psi \geq \lambda^{-2}$ implies that
$$
\pa_{\xp\xp} \psi \, \pa_{\xo \xo} \psi  \geq  \lambda^{-2},
$$
and the convexity of $\psi$ gives $\pa_{\xp\xp} \psi$,  $\pa_{\xo \xo} \psi \geq 0$.
We deduce
\begin{align*}
\left(\int_{-\frac{\ell}{2}}^{\frac{\ell}{2}} \lambda^{-1}\, d\xp\right)^2 
& \leq \left(\int_{-\frac{\ell}{2}}^{\frac{\ell}{2}} |\pa_{\xp\xp} \psi | ^{1/2} |\pa_{\xo\xo} \psi |^{1/2} dx_{\parallel}\right)^2 \\
&  \leq \int_{-\frac{\ell}{2}}^{\frac{\ell}{2}} \pa_{\xp\xp} \psi  \, d\xp \int_{-\frac{\ell}{2}}^{\frac{\ell}{2}} \pa_{\xo\xo} \psi  \, d\xp \\
& \leq \left[\pa_{\xp}\psi \left(\frac{\ell}{2}, x_{\perp}\right)- \pa_{\xp}\psi \left(-\frac{\ell}{2},x_{\perp}\right)\right] \int_{-\frac{\ell}{2}}^{\frac{\ell}{2}}\pa_{\xo\xo} \psi  dx_{\parallel},
\end{align*}
which implies (using \eqref{eq:psixp})
\[
\frac{\ell^3 }{4\lambda^2 K ( |x_{\perp}| + \gamma)} \leq \int_{-\frac{\ell}{2}}^{\frac{\ell}{2}} \pa_{\xo\xo} \psi dx_{\parallel} \mbox{.}
\]

Finally, integrating with respect to $x_{\perp}$ we get
\begin{align*}
\frac{\ell^3 }{4 \lambda^2 K} \int_{0}^{\delta} \frac{d x_{\perp}}{ x_{\perp} + \gamma } dx_{\perp}
&  \leq  \int_{-\frac{\ell}{2}}^{\frac{\ell}{2}} \int_{0}^{\delta} \pa_{\xo\xo} \psi \,  dx_{\perp} \, dx_{\parallel} \\ 
 & \leq \int_{-\frac{\ell}{2}}^{\frac{\ell}{2}} [\pa_{\xo} \psi (x_{\parallel},\delta)-\pa_{\xo} \psi (x_{\parallel},0)] \, dx_{\parallel} \\
& \leq 2 K \ell,
\end{align*}
and \eqref{eq:estcontinuous} follows.


\section{Proof of Theorem \ref{th:affinebound}}\label{app:2}
We have $1\leq d <n$ and we choose a system of coordinates 
$$x=(\xp,\xo)\in \RR^d\times \RR^{n-d} \quad \mbox{ with }  \xp=(x_1,\dots x_d) \mbox{  and } \xo=(x_{d+1},\dots x_n),$$
so that $H=\{\xo=0 \}$.
For $\ell<\delta/2$, we have $B^d_{\ell}(0) \times B^{n-d}_{\delta/2}(0)\subset \Omega$, and 
 the following lemma is the equivalent of Lemma 	\ref{lemma:SubdifferentialBound} in this higher dimensional setting (the proof is similar),
\begin{lemma}\label{lem:bd}
For all $\xp\in B^d_{\ell/2}(0)$ and for all $\xo\in B^{n-d}_{\delta/2}$ we have 
\begin{equation}\label{eq:bd} 
|\na_{\xp} \psi(\xp,\xo) | \leq \frac 2 \ell K( |\xo|+\gamma).
\end{equation}
\end{lemma}

The starting point of the proof of Theorem \ref{th:affinebound} is the following  consequence of
Fischer's inequality
\begin{equation}\label{eq:fisher}
 \det \left( D^2_{\xp} \psi\right)\det \left( D^2_{\xo} \psi \right)\geq \det ( D^2 \psi )\geq \lambda^{-2}.
\end{equation}
Integrating \eqref{eq:fisher} with respect to $\xp$ after taking the square root, we get for all $\xo\in B^{n-d}_{\delta/2}$
\begin{align*}
\lambda^{-1}  \mathcal L^d (B^d_{\ell/2} )
& \leq \int_{B_{\ell/2}^d} \left( \det D^2_{\xp} \psi\right)^{1/2}\left( \det D^2_{\xo} \psi \right)^{1/2}\, d\xp \\
& \leq \left(  \int_{B_{\ell/2}^d} \det D^2_{\xp} \psi\, d\xp \right)^{1/2}\left(\int_{B_{\ell/2}^d}  \det D^2_{\xo} \psi \, d\xp \right)^{1/2}\\
& \leq \left(  \mathcal L^d\left(\na_{\xp} \psi (B_{\ell/2}^d\times\{\xo\}\right)  \right)^{1/2}\left(\int_{B_{\ell/2}^d}  \det D^2_{\xo} \psi (\xp,\xo)\, d\xp	 \right)^{1/2},
\end{align*}
where we used the fact that for a convex function $\phi$, the integral $\int_U \det D^2\phi \, dx$ is the volume of the image of $U$ under $\nabla \phi$.

Using \eqref{eq:bd}  we deduce
$$
\lambda^{-2} \ell^{2d} \leq C\left( \frac 1\ell  K ( |\xo|+\gamma) \right)^d \int_{B_{\ell/2}^d}  \det D^2_{\xo} \psi(\xp,\xo) \, d\xp, 
$$
which implies in particular
\begin{align*}
\int_{B^{n-d}_{\delta/2}} \frac{ \lambda^{-2} \ell^{3d}}{ K^d ( |\xo|+\gamma )^d} \, d\xo 
& \leq C \int_{B^{n-d}_{\delta/2}}\int_{B_\ell^d}  \det D^2_{\xo} \psi \, d\xp d\xo\\
& \leq C \int_{B_\ell^d}  \int_{B^{n-d}_{\delta/2}} \det D^2_{x''} \psi \, d\xo d\xp\\
& \leq C \int_{B_\ell^d} \mathcal L^{n-d}\left( \na_{\xo} \psi(\{\xp \} \times B^{n-d}_{\delta/2})\right) \, d\xp\\
& \leq  C \ell^d K^{n-d}.
\end{align*}
We finally obtain that
\begin{equation} \label{eq:bd2}
  \ell^{2d} \int_{B^{n-d}_{\delta/2}} \frac{1}{( |\xo|+\gamma )^d} \, d\xo  \leq  C\lambda^2 K^n,
\end{equation}
where we can write
$$ \int_{B^{n-d}_{\delta/2}} \frac{1}{( |\xo|+\gamma )^d} \, d\xo  = \int_0^{\delta/2}   \frac{r^{n-d-1}}{( r +\gamma )^d}\, dr = \gamma^{n-2d}  \int_0^{\delta/2\gamma}   \frac{r^{n-d-1}}{( r +1 )^d}\, dr $$
and  \eqref{eq:estgammand} follows.


\bibliography{mybib}{}
\bibliographystyle{plain}
\end{document}